\documentclass{amsart}
\usepackage{amssymb}
\usepackage{mathrsfs}
\usepackage{cases}
\usepackage{amsmath}
\usepackage{amsfonts}
\usepackage{pifont}
\usepackage{graphicx,amssymb,mathrsfs,amsmath}
\usepackage{tocvsec2}
\newtheorem{thm}{Theorem}[section]

\newtheorem{prop}[thm]{Proposition}
\theoremstyle{definition}
\newtheorem{defn}[thm]{Definition}
\newtheorem{example}[thm]{Example}
\theoremstyle{remark}
\newtheorem{rem}[thm]{Remark}
\numberwithin{equation}{section}
\begin{document}
\title[Generalized weighted pseudo-almost automorphic...]{Generalized weighted pseudo-almost periodic solutions and generalized weighted pseudo-almost automorphic solutions of abstract Volterra integro-differential inclusions}

\author{Marko Kosti\' c}
\address{Faculty of Technical Sciences,
University of Novi Sad,
Trg D. Obradovi\' ca 6, 21125 Novi Sad, Serbia}
\email{marco.s@verat.net}

{\renewcommand{\thefootnote}{} \footnote{
\\ \text{  }  \ \    {\it Key words and phrases.} Weighted pseudo-almost periodicity, weighted pseudo-almost automorphy, convolution products, abstract Volterra integro-differential inclusions, multivalued linear operators.
\\  \text{  } The author is partially supported by grant 174024 of Ministry
of Science and Technological Development, Republic of Serbia.}}

\begin{abstract}
In this paper, we analyze the existence and uniqueness of generalized weighted pseudo-almost automorphic solutions of abstract Volterra integro-differential inclusions in Banach spaces. The main results are devoted to the study of various types of weighted pseudo-almost periodic (automorphic) properties of convolution products. We illustrate our abstract results with some examples and possible applications.
\end{abstract}
\maketitle

\section{Introduction and preliminaries}

The class of weighted pseudo-almost periodic functions with values in Banach spaces was introduced by T. Diagana in \cite{tokos1} (2006), while the class of weighted pseudo-almost automorphic functions with values in Banach spaces was introduced by J. Blot, G. M. Mophou, G. M. N'Gu\' er\' ekata and D. Pennequin in \cite{blot-weighted} (2009).
In this paper, we analyze the existence and uniqueness of generalized weighted pseudo-almost periodic solutions and generalized weighted pseudo-almost automorphic solutions of abstract Volterra integro-differential inclusions in Banach spaces. 

Concerning already made applications to abstract Volterra integro-differential equations, we will mention here only a few important results obtained so far. In \cite[Chapter 10]{diagana},
T. Diagana has examined the existence and uniqueness of weighted pseudo almost-periodic solutions of
autonomous partial differential
equation
\begin{align}\label{zvone-stepa}
\frac{d}{dt}\Bigl[ u^{\prime \prime}(t)+f(t,Bu(t)) \Bigr]=w(t)Au(t)+g(t,Cu(t)),\quad t\in {\mathbb R},
\end{align}
where $A$ is a sectorial operator on $X$, $B$ and $C$ are closed linear operators acting on $X$, and $f : {\mathbb R} \times  X \rightarrow X ,$  $g :{\mathbb R} \times  X \rightarrow X$ are pseudo almost-periodic functions in $t\in {\mathbb R}$ uniformly in $x\in X.$ The results on well-posedness of \eqref{zvone-stepa} clearly apply in the study of qualitative properties of solutions of the abstract nonautonomous third-order differential equation
\begin{align*}
u^{\prime \prime \prime}+B(t)u^{\prime}+A(t)u=h(t,u),\quad t\in {\mathbb R};
\end{align*}
see \cite[Chapter 10]{diagana} for more details. Motivated by the ideas of T. Diagana, a great number of authors has considered 
the qualitative properties of solutions for various types of the Sobolev-type differential equations. For example, in \cite{bian-filomat}, Y.-K. Chang and Y.-T. Bian have investigated
the weighted asymptotic behaviour of the Sobolev-type differential equation
$$
\frac{d}{dt}\Bigl[ u(t)+f(t,u(t)) \Bigr]=A(t)u(t)+g(t,u(t)),\quad t\in {\mathbb R},
$$
where $A(t) : D \subseteq X \rightarrow X$ is a family of densely defined closed linear operators on a common domain $D$,
independent of time $t\in {\mathbb R},$ $f : {\mathbb R} \times  X \rightarrow X$ is a weighted pseudo-almost automorphic function and $g :{\mathbb R} \times  X \rightarrow X$
is a Stepanov-like weighted pseudo-almost automorphic function satisfying some extra conditions. Here, it is worth noting that
the class of weighted Stepanov-like pseudo-almost
automorphic functions has been introduced by Z. Xia and M. Fan in \cite{stepanoff-xia}, where the authors have analyzed the existence and uniqueness of such solutions for the following abstract semilinear integro-differential equation:
$$
u(t)=g(t)+\int^{t}_{-\infty}a(t-s)f(s,u(s))\, ds,\quad t\in {\mathbb R},
$$
under certain conditions.
In \cite{abbas-indian},
S. Abbas, V. Kavitha and R. Murugesu have examined Stepanov-like weighted pseudo almost automorphic solutions
to the following fractional order abstract integro-differential equation:
$$
D_{t}^{\alpha}u(t)=Au(t)+D_{t}^{\alpha -1}f(t,u(t),Ku(t)),\quad t\in {\mathbb R},
$$
where
$
Ku(t)=\int^{t}_{-\infty}k(t-s)h(s,u(s))\, ds,$  $t\in {\mathbb R},
$
$1<\alpha<2,$ $A$ is a sectorial operator with domain and range in $X,$ of negative sectorial type $\omega<0,$ the function $k(t)$ is exponentially decaying,
the functions $f : {\mathbb R} \times X \times X \rightarrow X$ and
$h : {\mathbb R} \times X \rightarrow X$ are Stepanov-like weighted pseudo-almost automorphic in time for
each fixed elements of $X\times X$ and $X,$ respectively, satisfying some extra conditions (cf. also T. Diagana \cite{diagana-stepanoff}).

Weighted pseudo-almost automorphic solutions to functional differential equations with infinite delay have been recently considered by Y.-K. Chang and S. Zheng in \cite{chang-ejde}. They have analyzed the abstract nonautonomous differential equation
$$
\frac{d}{dt}\Bigl[ u(t)+f(t,u_{t}) \Bigr]=A(t)u(t)+g(t,u_{t}),\quad t\in {\mathbb R},
$$
where $A(t) : D \subseteq X \rightarrow X$ is a family of densely defined closed linear operators on a common domain $D$,
independent of time $t\in {\mathbb R},$ 
the history $u_{t} : (0,\infty] \rightarrow X$ defined by $u_{t}(\cdot):=u(t+\cdot)$
belongs to some abstract phase space ${\mathcal B}$ defined
axiomatically, and 
$f : {\mathbb R} \times {\mathcal B} \rightarrow X,$ $g :{\mathbb R} \times  {\mathcal B} \rightarrow X$
fulfill some conditions.

Fractional calculus and fractional differential equations have received much
attention recently due to its wide range of applications in various fields of science, such as mathematical physics,
engineering, aerodynamics, biology, chemistry, economics etc (see e.g. \cite{bajlekova}, \cite{Diet}, \cite{kilbas}-\cite{knjigaho} and \cite{prus}-\cite{samko}). In this paper, we essentially use only
the Weyl-Liouville fractional derivatives (for more details, we refer the reader to the paper \cite{relaxation-peng}
by J. Mu, Y. Zhoa and L. Peng). The Weyl-Liouville fractional derivative
$D_{t,+}^{\gamma}u(t)$ of order $\gamma \in (0,1)$ is defined for those continuous functions
$u : {\mathbb R} \rightarrow X$
such that $t\mapsto \int_{-\infty}^{t}g_{1-\gamma}(t-s)u(s)\, ds,$ $t\in {\mathbb R}$ is a well-defined continuously differentiable mapping, by
$$
D_{t,+}^{\gamma}u(t):=\frac{d}{dt}\int_{-\infty}^{t}g_{1-\gamma}(t-s)u(s)\, ds,\quad t\in {\mathbb R};
$$
here and hereafter, $g_{\zeta}(t):=t^{\zeta-1}/\Gamma(\zeta),$ $t>0,$ where $\Gamma (\cdot)$ denotes the Gamma function ($\zeta>0$).

The organization of paper is described as follows. In Subsection \ref{sajt-mlo}, we present a brief overview of definitions and results about multivalued linear operators in Banach spaces. The main aim of Section \ref{section1} is to analyze various generalizations of almost periodic functions and almost automorphic functions; besides some new notions introduced, like Besicovitch pseudo-almost automorphy, the only original contribution here is Proposition \ref{ramanujan} in which we show that the class of
Weyl-$p$-vanishing functions (extended by zero outside $[0,\infty);$ see \cite{weyl} for the notion) is contained in the Weyl ergodic type space $W_{p}PAA_{0}({\mathbb R} : X),$ introduced by S. Abbas in \cite{sead-abas}; we paid special attention to generalized two-parameter almost automorphic functions. The main aim of Section \ref{weighted-primonja} is to inscribe the basic facts about weighted pseudo-almost periodic functions, weighted pseudo-almost automorphic functions and their generalizations; in Subsection \ref{PIVOTI}, we repeat some known composition principles for these classes. In Section \ref{weighted-primonja}, we introduce the notions of a weighted Weyl $p$-pseudo almost automorphic function and a weighted Besicovitch $p$-pseudo almost automorphic function, investigating also 
the convolution invariance of space $B^{1}WPAA_{0}({\mathbb R},X, \rho_{1}, \rho_{2})$ and translation invariance of space 
$PAP_{0}({\mathbb R}, X,\rho_{1},\rho_{2});$ see Section \ref{section1}-Section \ref{weighted-primonja} for the notions.
Our main results are stated in Section \ref{konvolucije}, where we investigate the generalized weighted pseudo-almost periodic (automorphic) properties of convolution products; in this paper, we are primarily concerned with the infinite convolution product 
$\int^{\cdot}_{-\infty}R(\cdot-s)f(s)\, ds,$ where $(R(t))_{t>0}$ is a strongly continuous operator family.
Weighted pseudo-almost automorphic solutions of semilinear (fractional) Cauchy inclusions are analyzed in Section \ref{semilinear-spring}, where we also provide some applications in the analysis of the following abstract Cauchy inclusion of first order
\begin{align*}
u^{\prime}(t)\in {\mathcal A}u(t)+f(t),\ t\in {\mathbb R},
\end{align*}
and its fractional relaxation analogue
\begin{align*}
D_{t,+}^{\gamma}u(t)\in {\mathcal A}u(t)+f(t),\ t\in {\mathbb R},
\end{align*}
where $D_{t,+}^{\gamma}$ denotes the Riemann-Liouville fractional derivative of order $\gamma \in (0,1),$ and
$f : {\mathbb R} \rightarrow X$ satisfies certain properties,
as well as their semilinear analogues
\begin{align}\label{favini}
u^{\prime}(t)\in {\mathcal A}u(t)+f(t,u(t)),\ t\in {\mathbb R},
\end{align}
and
\begin{align}\label{left}
D_{t,+}^{\gamma}u(t)\in {\mathcal A}u(t)+f(t,u(t)),\ t\in {\mathbb R},
\end{align}
where
$f : {\mathbb R} \times X \rightarrow X$ satisfies certain applications. Here, ${\mathcal A}$ is a closed multivalued linear operator.  
Before we switch to our first subsection, it should be emphasized that our results seem to be new even for a class of almost sectorial operators \cite{pb1}. 

\subsection{Multivalued linear operators}\label{sajt-mlo}

In this subsection, we will present some necessary definitions from the theory of multivalued
linear operators. For more details about this topic, we refer the reader to the monographs by R. Cross \cite{cross} and A. Favini, A. Yagi \cite{faviniyagi}.

Let $X$ and $Y$ be two Banach spaces over the field of complex numbers.
A multivalued map ${\mathcal A} : X \rightarrow P(Y)$ is said to be a multivalued
linear operator (MLO) iff the following two conditions hold:
\begin{itemize}
\item[(i)] $D({\mathcal A}) := \{x \in X : {\mathcal A}x \neq \emptyset\}$ is a linear submanifold of $X$;
\item[(ii)] ${\mathcal A}x +{\mathcal A}y \subseteq {\mathcal A}(x + y),$ $x,\ y \in D({\mathcal A})$
and $\lambda {\mathcal A}x \subseteq {\mathcal A}(\lambda x),$ $\lambda \in {\mathbb C},$ $x \in D({\mathcal A}).$
\end{itemize}
If $X=Y,$ then we say that ${\mathcal A}$ is an MLO in $X.$
We know that the equality $\lambda {\mathcal A}x + \eta {\mathcal A}y = {\mathcal A}(\lambda x + \eta y)$ holds
for every $x,\ y\in D({\mathcal A})$ and for every $\lambda,\ \eta \in {\mathbb C}$ with $|\lambda| + |\eta| \neq 0.$ If ${\mathcal A}$ is an MLO, then ${\mathcal A}0$ is a linear submanifold of $Y$
and ${\mathcal A}x = f + {\mathcal A}0$ for any $x \in D({\mathcal A})$ and $f \in {\mathcal A}x.$ Set $R({\mathcal A}):=\{{\mathcal A}x :  x\in D({\mathcal A})\}.$
Then the set $N({\mathcal A}):={\mathcal A}^{-1}0 = \{x \in D({\mathcal A}) : 0 \in {\mathcal A}x\}$ is called the kernel space
of ${\mathcal A}.$ The inverse ${\mathcal A}^{-1}$ of an MLO is defined through
$D({\mathcal A}^{-1}) := R({\mathcal A})$ and ${\mathcal A}^{-1} y := \{x \in D({\mathcal A}) : y \in {\mathcal A}x\}$.
It can be easily verified that ${\mathcal A}^{-1}$ is an MLO in $X,$ as well as that $N({\mathcal A}^{-1}) = {\mathcal A}0$
and $({\mathcal A}^{-1})^{-1}={\mathcal A}.$ If $N({\mathcal A}) = \{0\},$ i.e., if ${\mathcal A}^{-1}$ is
single-valued, then ${\mathcal A}$ is said to be injective. Assuming ${\mathcal A},\ {\mathcal B} : X \rightarrow P(Y)$ are two MLOs, then we define its sum ${\mathcal A}+{\mathcal B}$ by $D({\mathcal A}+{\mathcal B}) := D({\mathcal A})\cap D({\mathcal B})$ and $({\mathcal A}+{\mathcal B})x := {\mathcal A}x +{\mathcal B}x,$ $x\in D({\mathcal A}+{\mathcal B}).$
It is clear that ${\mathcal A}+{\mathcal B}$ is an MLO. We write ${\mathcal A} \subseteq {\mathcal B}$ iff $D({\mathcal A}) \subseteq D({\mathcal B})$ and ${\mathcal A}x \subseteq {\mathcal B}x$
for all $x\in D({\mathcal A}).$ Products, integer powers and multiplication with scalar constants are well-known operations for MLOs (\cite{faviniyagi}).

We say that an MLO ${\mathcal A} : X\rightarrow P(Y)$ is closed iff for any
sequences $(x_{n})$ in $D({\mathcal A})$ and $(y_{n})$ in $Y$ such that $y_{n}\in {\mathcal A}x_{n}$ for all $n\in {\mathbb N}$ we have that $\lim_{n \rightarrow \infty}x_{n}=x$ and
$\lim_{n \rightarrow \infty}y_{n}=y$ imply
$x\in D({\mathcal A})$ and $y\in {\mathcal A}x.$

Suppose now that ${\mathcal A}$ is an MLO in $X.$ The resolvent set of ${\mathcal A},$ $\rho({\mathcal A})$ for short, is defined as the union of those complex numbers
$\lambda \in {\mathbb C}$ for which
\begin{itemize}
\item[(i)] $X= R(\lambda-{\mathcal A})$;
\item[(ii)] $(\lambda - {\mathcal A})^{-1}$ is a single-valued linear continuous operator on $X.$
\end{itemize}
The operator $\lambda \mapsto (\lambda -{\mathcal A})^{-1}$ is called the resolvent of ${\mathcal A}$ ($\lambda \in \rho({\mathcal A})$). Any MLO with non-empty resolvent set is closed and the Hilbert resolvent equation holds in our framework.
For further information about multivalued linear operators and abstract degenerate integro-differential equations, we refer the reader to the monographs \cite{carol}, \cite{cross}, \cite{faviniyagi}, \cite{FKP}, \cite{me152} and \cite{svir-fedorov}.

\section{Various generalizations of almost periodic functions and almost automorphic functions}\label{section1}

The concept of almost periodicity was introduced by Danish mathematician H. Bohr around 1924-1926 and later generalized by many other authors (cf. \cite{diagana}, \cite{gaston}, \cite{nova-mono} and \cite{30} for more details on the subject).
Let $I={\mathbb R}$ or $I=[0,\infty),$ and let $f : I \rightarrow X$ be continuous. Given $\epsilon>0,$ we call $\tau>0$ an $\epsilon$-period for $f(\cdot)$ iff
$
\| f(t+\tau)-f(t) \| \leq \epsilon,$ $t\in I.
$
The set constituted of all $\epsilon$-periods for $f(\cdot)$ is denoted by $\vartheta(f,\epsilon).$ It is said that $f(\cdot)$ is almost periodic, a.p. for short, iff for each $\epsilon>0$ the set $\vartheta(f,\epsilon)$ is relatively dense in $I,$ which means that
there exists $l>0$ such that any subinterval of $I$ of length $l$ meets $\vartheta(f,\epsilon)$.

Let $f : {\mathbb R} \rightarrow X$ be continuous. Then we say that
$f(\cdot)$ is almost automorphic, a.a. for short, iff for every real sequence $(b_{n})$ there exist a subsequence $(a_{n})$ of $(b_{n})$ and a map $g : {\mathbb R} \rightarrow X$ such that
\begin{align}\label{first-equ}
\lim_{n\rightarrow \infty}f\bigl( t+a_{n}\bigr)=g(t)\ \mbox{ and } \  \lim_{n\rightarrow \infty}g\bigl( t-a_{n}\bigr)=f(t),
\end{align}
pointwise for $t\in {\mathbb R}.$ In this case, we have that $f\in C_{b}({\mathbb R} : X)$ and the limit function $g(\cdot)$ must be bounded on ${\mathbb R}$ but not necessarily continuous on ${\mathbb R}.$ Furthermore, it is clear that the uniform convergence of one of the limits appearing in \eqref{first-equ} implies the convergenece of the second one in this equation and that, in this case, the function
$f(\cdot)$ has to be almost periodic and the function $g(\cdot)$ has to be continuous.  In the case that the convergence of limits appearing in \eqref{first-equ} is uniform on compact subsets of ${\mathbb R},$ then we say that $f(\cdot)$ is compactly almost automorphic, c.a.a. for short. The vector space consisting of all almost automorphic, resp., compactly almost automorphic functions, is denoted by $AA({\mathbb R} :X),$ resp., $AA_{c}({\mathbb R} :X).$ By Bochner's criterion \cite{diagana}, any almost periodic function has to be compactly almost automorphic. The converse statement is not true, however (\cite{diagana}, \cite{nova-mono}).
It is well-known that the range of an almost automorphic function needs to be a relatively compact subset of $X.$

A continuous function $f : {\mathbb R} \rightarrow X$ is said to be
asymptotically (compact) almost automorphic, a.(c.)a.a. for short, iff there are a function $h\in C_{0}([0,\infty) : X)$ and a (compact) almost automorphic function $q: {\mathbb R} \rightarrow X$ such that
$f(t)=h(t)+q(t),$ $t\geq 0.$ Using Bochner's criterion again, we can easily show that any asymptotically almost periodic function $[0,\infty) \mapsto X$ is asymptotically (compact) almost automorphic. The spaces of almost periodic, almost automorphic, compactly almost automorphic functions, and asymptotically (compact) almost automorphic functions are closed subspaces of $C_{b}({\mathbb R} : X)$ when equipped with the sup-norm. %For the sake of simplicity of notation, we will denote these Banach spaces by their support sets.

Suppose now that $1\leq p <\infty,$ $l>0$ and $f,\ g\in L^{p}_{loc}(I :X),$ where $I={\mathbb R}$ or $I=[0,\infty).$ Define the Stepanov `metric' by
\begin{align*}
D_{S_{l}}^{p}\bigl[f(\cdot),g(\cdot)\bigr]:= \sup_{x\in I}\Biggl[ \frac{1}{l}\int_{x}^{x+l}\bigl \| f(t) -g(t)\bigr\|^{p}\, dt\Biggr]^{1/p}.
\end{align*}
Then there exists 
\begin{align}\label{weyl-metric}
D_{W}^{p}\bigl[f(\cdot),g(\cdot)\bigr]:=\lim_{l\rightarrow \infty}D_{S_{l}}^{p}\bigl[f(\cdot),g(\cdot)\bigr]
\end{align}
in $[0,\infty].$ The distance appearing in \eqref{weyl-metric} is called the Weyl distance of $f(\cdot)$ and $g(\cdot).$ The Stepanov and Weyl `norm' of $f(\cdot)$ are defined by
$$
\bigl\| f  \bigr\|_{S_{l}^{p}}:= D_{S_{l}}^{p}\bigl[f(\cdot),0\bigr]\ \mbox{  and  }\ \bigl\| f  \bigr\|_{W^{p}}:= D_{W}^{p}\bigl[f(\cdot),0\bigr],
$$
respectively.

We say that a function $f\in L^{p}_{loc}(I :X)$ is Stepanov $p$-bounded, $S^{p}$-bounded shortly, iff
$$
\|f\|_{S^{p}}:=\sup_{t\in I}\Biggl( \int^{t+1}_{t}\|f(s)\|^{p}\, ds\Biggr)^{1/p}<\infty.
$$
The above norm turns the space $L_{S}^{p}(I:X)$ consisting of all $S^{p}$-bounded functions into a Banach space.
It is said that a function $f\in L_{S}^{p}(I:X)$ is Stepanov $p$-almost periodic, $S^{p}$-almost periodic or $S^{p}$-a.p. shortly, iff the function
$
\hat{f} : I \rightarrow L^{p}([0,1] :X),
$ defined by
$
\hat{f}(t)(s):=f(t+s),$ $t\in I,\ s\in [0,1]
$
is almost periodic.
It is said that $f\in  L_{S}^{p}([0,\infty): X)$ is asymptotically Stepanov $p$-almost periodic, asymptotically $S^{p}$-a.p. shortly, iff $\hat{f} : [0,\infty) \rightarrow L^{p}([0,1]:X)$ is asymptotically almost periodic.

The following notion of a Stepanov $p$-almost automorphic function has been introduced by G. M. N'Gu\' er\' ekata and A. Pankov in \cite{gaston-pankov}: A function $f\in L_{loc}^{p}({\mathbb R}:X)$ is called Stepanov $p$-almost automorphic, $S^{p}$-almost automorphic or $S^{p}$-a.a. shortly, iff for
every real sequence $(a_{n}),$ there exists a subsequence $(a_{n_{k}})$
and a function $g\in L_{loc}^{p}({\mathbb R}:X)$ such that
\begin{align*}%\label{ne-mu-je-stepanov}
\lim_{k\rightarrow \infty}\int^{t+1}_{t}\Bigl \| f\bigl(a_{n_{k}}+s\bigr) -g(s)\Bigr \|^{p} \, ds =0
\end{align*}
and
\begin{align*}
\lim_{k\rightarrow \infty}\int^{t+1}_{t}\Bigl \| g\bigl( s-a_{n_{k}}\bigr) -f(s)\Bigr \|^{p} \, ds =0
\end{align*}
for each $ t\in {\mathbb R}$; a function $f\in L_{loc}^{p}([0,\infty):X)$ is said to be asymptotically Stepanov $p$-almost automorphic, asymptotically $S^{p}$-almost automorphic or asymptotically $S^{p}$-a.a. shortly, iff there exists an $S^{p}$-almost automorphic
function $g(\cdot)$ and a function $q\in  L_{S}^{p}([0,\infty): X)$ such that $f(t)=g(t)+q(t),$ $t\geq 0$ and $\hat{q}\in C_{0}([0,\infty) : L^{p}([0,1]: X));$ the vector space consisting of such functions $q(\cdot)$ will be denoted by $S^{p}_{0}([0,\infty) : X).$ It is well known that the $S^{p}$-almost automorphy of $f(\cdot)$ implies the compact almost automorphy of the mapping
$
\hat{f} : I \rightarrow L^{p}([0,1] :X)
$ defined above, with the limit function being $g(\cdot)(s):=g(s+\cdot)$ for a.e. $s\in [0,1],$ so that any $S^{p}$-almost automorphic function $f(\cdot)$ has to be $S^{p}$-bounded ($1\leq p<\infty$). The vector space consisting of all $S^{p}$-almost automorphic functions, resp., asymptotically $S^{p}$-almost automorphic functions, will be denoted by $ AAS^{p}({\mathbb R} : X),$ resp., $ AAAS^{p}([0,\infty) : X).$

If $1\leq p<q<\infty$ and $f(\cdot)$ is Stepanov $q$-almost automorphic, resp., Stepanov $q$-almost periodic, then $f(\cdot)$ is Stepanov $p$-almost automorphic, resp., Stepanov $p$-almost periodic. Furthermore, the (asymptotical) Stepanov $p$-almost periodicity of $f(\cdot)$ for some $p\in [1,\infty)$ implies the (asymptotical) Stepanov $p$-almost automorphy of $f(\cdot).$
It is a well-known fact that if $f(\cdot)$ is an almost periodic (respectively, a.a.p., a.a., a.a.a.) function
then $f(\cdot)$ is also $S^p$-almost periodic (resp., asymptotically $S^p$-a.p.,  $S^p$-a.a., asymptotically $S^p$-a.a.) for $1\leq p <\infty.$ The converse statement is not true, in general.

We refer the reader to \cite{weyl} for the notions of an (equi-)Weyl-$p$-almost periodic function and an 
(equi-)Weyl-$p$-vanishing function ($1\leq p <\infty$); see also \cite{deda} for scalar-valued case.
Denote by $W^{p}_{0}([0,\infty):X)$ and $e-W^{p}_{0}([0,\infty):X)$ the sets consisting of all Weyl-$p$-vanishing functions and equi-Weyl-$p$-vanishing functions, respectively. Then we know that $e-W^{p}_{0}([0,\infty):X)\subseteq W^{p}_{0}([0,\infty):X).$

The concepts of Weyl almost
automorphy and Weyl pseudo-almost automorphy, more general than those of
Stepanov almost automorphy and Stepanov pseudo-almost automorphy, were introduced by S. Abass \cite{sead-abas} in 2012. We will first give the definition of a Weyl
$p$-almost automorphic function.

\begin{defn}\label{stea-weyl}
Let $p\geq 1.$ Then we say that a function $f\in L_{loc}^{p}({\mathbb R} : X)$ is Weyl
$p$-almost automorphic iff for
every real sequence $(s_{n}),$ there exist a subsequence $(s_{n_{k}})$
and a function $f^{\ast} \in L^{p}
_{loc}({\mathbb R} : X)$ such that
\begin{align*}
\lim_{k\rightarrow \infty} \lim_{l\rightarrow +\infty}\frac{1}{2l}\int^{l}_{-l}\Bigl \| f\bigl( t+s_{n_{k}}+x\bigr) -f^{\ast}(t+x)\Bigr \|^{p} \, dx =0
\end{align*}
and
\begin{align*}
\lim_{k\rightarrow \infty} \lim_{l\rightarrow +\infty}\frac{1}{2l}\int^{l}_{-l}\Bigl \| f^{\ast}\bigl( t-s_{n_{k}}+x\bigr) -f(t+x)\Bigr \|^{p} \, dx =0
\end{align*}
for each $ t\in {\mathbb R}.$ The set of all such functions are denoted
by $W^{p}AA({\mathbb R} : X).$
\end{defn}

The set $W^{p}AA({\mathbb R} : X),$
equipped with the usual operations of pointwise addition of functions and multiplication of functions with scalars, forms a vector space.  Weyl-$p$-almost periodic functions forms a linear subspace of $W^{p}AA({\mathbb R} : X).$

The class of Stepanov pseudo-almost automorphic functions has been introduced by T. Diagana in \cite{diagana-definicija} (cf. also the paper \cite{xl-fany}, where Z. Fan, J. Liang and T.-J. Xiao have analyzed composition theorems for the corresponding class):

\begin{defn}\label{novi-weighted-auto-stepa-aw} 
Let $1\leq p<\infty.$ A Stepanov $p$-bounded function $f(\cdot)$ is said to be Stepanov $p$-pseudo almost periodic (automorphic) iff it admits a decomposition
$
f(t)=g(t)+q(t),
$ $t\in {\mathbb R},$ where $g(\cdot)$ is $S^{p}$-almost periodic (automorphic) and $q(\cdot) \in L_{loc}^{p}({\mathbb R} : X)$ satisfies
\begin{align}\label{stepa-auto-durak}
\lim_{T\rightarrow +\infty}\frac{1}{2T}\int^{T}_{-T}\Biggl[ \int^{t+1}_{t}\|q(s)\|^{p}\, ds \Biggr]^{1/p} \, dt =0.
\end{align}
Denote by $S^{p}PAP({\mathbb R} :X)$ ($S^{p}PAA({\mathbb R} :X)$) the collection of such functions,
and by $S^{p}PAA_{0}({\mathbb R} : X)$ the collection of locally $p$-integrable $X$-valued functions $q(\cdot)$ such that \eqref{stepa-auto-durak} holds.
\end{defn}

It is well known that $S^{p}PAA({\mathbb R},X, \rho_{1}, \rho_{2})$, endowed with the Stepanov norm, is a Banach space which contains all pseudo-almost automorphic functions. A similar statement holds for the space $S^{p}PAP({\mathbb R},X, \rho_{1}, \rho_{2}).$

\begin{defn}\label{stea-weyl-abas} \cite{sead-abas}
Let $p\geq 1.$ Then we say that a function $f\in L_{loc}^{p}({\mathbb R} : X)$ is Weyl 
$p$-pseudo almost automorphic iff $f(\cdot)=g(\cdot)+q(\cdot),$ where $g(\cdot)$ is Weyl
$p$-almost automorphic and $q \in L^{p}
_{loc}({\mathbb R} : X)$ satisfies
\begin{align}\label{ne-mu-je-sd-sead}
 \lim_{T\rightarrow +\infty}\frac{1}{2T}\int^{T}_{-T}\Biggl[ \lim_{l\rightarrow +\infty}\frac{1}{2l}\int^{x+l}_{x-l}\bigl\| q(t)\|^{p}\, dt\Biggr]^{1/p} \, dx =0.
\end{align}
The set of all such functions are denoted
by $W_{p}PAA({\mathbb R} : X),$ while the set consisting of all functions $q \in L^{p}
_{loc}({\mathbb R} : X)$ satisfying \eqref{ne-mu-je-sd-sead} is denoted by $W_{p}PAA_{0}({\mathbb R} : X).$
\end{defn}

Now we will prove that the class of Weyl-$p$-vanishing functions (extended by zero outside $[0,\infty)$) is contained in $W_{p}PAA_{0}({\mathbb R} : X):$ 

\begin{prop}\label{ramanujan}
Let $1\leq p <\infty$, and let $q\in L^{p}_{loc}([0,\infty) : X)$ be a Weyl-$p$-vanishing function. Let $q_{e} \in L^{p}_{loc}({\mathbb R}: X)$ be given by $q_{e}(t):=q(t),$ $t\geq 0$ and $q_{e}(t):=0,$ $t<0.$  Then $q_{e}\in W_{p}PAA_{0}({\mathbb R} : X).$ 
\end{prop}

\begin{proof}
We only need to prove that \eqref{ne-mu-je-sd-sead} holds with $q(\cdot)$ replaced therein with $q_{e}(\cdot)$, i.e., that
\begin{align*}
 \lim_{T\rightarrow +\infty}\frac{1}{2T}\int^{T}_{0}\Biggl[ \lim_{l\rightarrow +\infty}\frac{1}{2l}\int^{x+l}_{0}\bigl\| q(t)\|^{p}\, dt\Biggr]^{1/p} \, dx =0.
\end{align*}
Let $x\in [0,T]$ be fixed. It suffices to show that
\begin{align}\label{belezore}
 \lim_{l\rightarrow +\infty}\frac{1}{2l}\int^{x+l}_{0}\bigl\| q(t)\|^{p}\, dt=0.
\end{align}
Towards this end, fix a number $\epsilon>0.$ Owing to the fact that  $q(\cdot)$ is Weyl-$p$-vanishing, i.e., that
$$
\lim_{t\rightarrow \infty}\, \lim_{l\rightarrow \infty}\ \ \sup_{x\geq t}\Biggl[ \frac{1}{l}\int_{x}^{x+l}\bigl \| q(s)\bigr\|^{p}\, ds\Biggr]^{1/p}=0,
$$
we have the existence of numbers $l_{0}(\epsilon)>0$ and $t_{0}(\epsilon)>0$ such that
\begin{align}\label{qwer-zore}
\frac{1}{2l}\int_{x}^{x+l}\bigl \| q(s)\bigr\|^{p}\, ds \leq \epsilon,\quad x\geq t_{0}(\epsilon),\ l\geq l_{0}(\epsilon).
\end{align}
Then we have
\begin{align*}
\frac{1}{2l}\int^{x+l}_{0}\bigl\| q(t)\|^{p}\, dt \leq \frac{1}{2l}\Biggl[ \int^{x}_{0}\|q(s)\|^{p}\, ds + \int^{x+l}_{x}\|q(s)\|^{p}\, ds \Biggr],\quad l>0.
\end{align*}
If $x\geq t_{0}(\epsilon),$ then the addend  $(1/2l)\int^{x+l}_{x}\|q(s)\|^{p}\, ds$ is less or equal than $\epsilon$ by \eqref{qwer-zore}, which clearly implies the existence of a number $l_{1}(\epsilon)>0$ such that for each $l\geq l_{1}(\epsilon)$ we have 
\begin{align*}
\frac{1}{2l}\int^{x+l}_{0}\bigl\| q(t)\|^{p}\, dt \leq 2\epsilon .
\end{align*}
If $x< t_{0}(\epsilon),$ then we have
\begin{align*}
\frac{1}{2l}\int^{x+l}_{0}\bigl\| q(t)\|^{p}\, dt & \leq \frac{1}{2l}\Biggl[ \int^{ t_{0}(\epsilon)}_{x}\|q(s)\|^{p}\, ds +\int^{ t_{0}(\epsilon)+l}_{ t_{0}(\epsilon)}\|q(s)\|^{p}\, ds \Biggr]
\\ & \leq\frac{1}{2l}\int^{ t_{0}(\epsilon)}_{x}\|q(s)\|^{p}\, ds +\epsilon,\quad l \geq l_{1}(\epsilon),
\end{align*}
which clearly implies the existence of a number $l_{2}(\epsilon)>l_{1}(\epsilon)$ such that for each $l\geq l_{2}(\epsilon)$ we have 
\begin{align*}
\frac{1}{2l}\int^{x+l}_{0}\bigl\| q(t)\|^{p}\, dt \leq 2\epsilon .
\end{align*}
This yields \eqref{belezore} and completes the proof of proposition.
\end{proof}

For Besicovitch and Besicovitch-Doss generalizations of almost periodic functions in Banach spaces, we refer the reader to \cite{NSJOM-besik}; see also \cite{besik}.
The class of Besicovitch almost automorphic functions has been introduced by F. Bedouhene, N. Challali, O. Mellah, P. Raynaud de Fitte and M. Smaali in \cite{besik-fatajou}. This class extends the class of
Weyl almost automorphic functions and its definition is given as follows:

\begin{defn}\label{stea-weyl}
Let $p\geq 1.$ Then we say that a function $f\in L_{loc}^{p}({\mathbb R} : X)$ is Besicovitch
$p$-almost automorphic iff for
every real sequence $(s_{n}),$ there exist a subsequence $(s_{n_{k}})$
and a function $f^{\ast} \in L^{p}
_{loc}({\mathbb R} : X)$ such that
\begin{align*}
\lim_{k\rightarrow \infty} \limsup_{l\rightarrow +\infty}\frac{1}{2l}\int^{l}_{-l}\Bigl \| f\bigl( t+s_{n_{k}}+x\bigr) -f^{\ast}(t+x)\Bigr \|^{p} \, dx =0
\end{align*}
and
\begin{align*}
\lim_{k\rightarrow \infty} \limsup_{l\rightarrow +\infty}\frac{1}{2l}\int^{l}_{-l}\Bigl \| f^{\ast}\bigl( t-s_{n_{k}}+x\bigr) -f(t+x)\Bigr \|^{p} \, dx =0
\end{align*}
for each $ t\in {\mathbb R}.$ The set of all such functions are denoted
by $B^{p}AA({\mathbb R} : X).$
\end{defn}

It can be easily proved that the set $B^{p}AA({\mathbb R} : X),$
equipped with the usual operations, forms a vector space. In the present situation, the author does not know whether
a Besicovitch $p$-almost periodic function is necessarily Besicovitch $p$-almost automorphic.

Now we will introduce the notion of a Besicovitch 
$p$-pseudo almost automorphic function:

\begin{defn}\label{stea-weyl-abas}
Let $p\geq 1.$ Then we say that a function $f\in L_{loc}^{p}({\mathbb R} : X)$ is Besicovitch
$p$-pseudo almost automorphic iff $f(\cdot)=g(\cdot)+q(\cdot),$ where $g(\cdot)$ is Besicovitch
$p$-almost automorphic and $q \in L^{p}
_{loc}({\mathbb R} : X)$ satisfies
\begin{align*}
 \lim_{T\rightarrow +\infty}\frac{1}{2T}\int^{T}_{-T}\Biggl[ \limsup_{l\rightarrow +\infty}\frac{1}{2l}\int^{x+l}_{x-l}\bigl\| q(t)\|^{p}\, dt\Biggr]^{1/p} \, dx =0.
\end{align*}
The set of all such functions are denoted
by $B_{p}PAA({\mathbb R} : X).$
\end{defn}

A very simple analysis shows that $B_{p}PAA({\mathbb R} : X)$ is a  vector space, as well as that $W_{p}PAA({\mathbb R} : X)\subseteq B_{p}PAA({\mathbb R} : X).$

In order to shorten our study, we will skip here all details concerning two-parameter almost periodic functions and their generalizations; see T. Diagana \cite{diagana} and M. Kosti\' c \cite{nova-mono} for further information in this direction. Regarding generalized two-parameter almost automorphic functions, the following facts should be recalled:
Let $(Y,\|\cdot \|_{Y})$ be another pivot space over the field of complex numbers.
Then we say that a continuous function $F : {\mathbb R} \times Y \rightarrow X$ is almost
automorphic iff
for every sequence of real numbers $(s_{n}')$ there exists a subsequence $(s_{n})$ such
that
$
G(t, y) := \lim_{n\rightarrow \infty}F(t +s_{n}, y)
$
is well defined for each $t \in {\mathbb R}$ and $y\in Y,$ and
$
\lim_{n\rightarrow \infty}G(t -s_{n}, y) = F(t, y)
$
for each $t \in {\mathbb R}$ and $y\in Y.$ The vector space consisting of such functions will be denoted by $AA({\mathbb R} \times Y :X).$

The notion of a pseudo almost-automorphic function was introduced by T.-J. Xiao, J. Liang and J. Zhang  \cite{xlz-pseudo} in 2008. Let us recall that the space of pseudo-almost automorphic functions, denoted shortly by $PAA({\mathbb R} : X),$ is defined as the direct sum of spaces $AA({\mathbb R} : X)$
and $PAP_{0}({\mathbb R} : X),$ where
$PAP_{0}({\mathbb R} : X)$ denotes the space consisting of all bounded continuous functions $\Phi : {\mathbb R} \rightarrow X$ such
that
$$
\lim_{r\rightarrow \infty} \frac{1}{2r}\int^{r}_{-r}\| \Phi(s)\|\, ds=0.
$$
Equipped with the sup-norm, the space $PAA({\mathbb R} : X)$ becomes one of Banach's. A bounded continuous function $f :  {\mathbb R} \times Y \rightarrow X$ is said to be
pseudo-almost automorphic iff $F = G +\Phi,$ where $G \in AA({\mathbb R} \times Y :X)$ and $\Phi \in
PAP_{0}({\mathbb R} \times Y:X);$ here, $PAP_{0}({\mathbb R} \times Y : X)$ denotes the space consisting of all continuous functions $\Phi : {\mathbb R} \times Y \rightarrow X$ such
that $\{\Phi (t,y) : t \in {\mathbb R} \} $ is bounded for all $y\in Y,$ and
$$
\lim_{r\rightarrow \infty} \frac{1}{2r}\int^{r}_{-r}\| \Phi(s,y)\|\, ds=0,
$$
uniformly in $y\in Y.$ The collection of such functions will be denoted henceforth by $PAA({\mathbb R} \times Y:X).$

\begin{defn}\label{stepanov-auto-wr} (see e.g. \cite{opuscula})
Let $1\leq p<\infty,$ and let $f : {\mathbb R} \times Y \rightarrow X$ satisfy that for each $y\in Y$ we have $f(\cdot,y)\in L^{p}_{loc}({\mathbb R} : X).$ Then it is said that $f(\cdot,\cdot)$ is Stepanov $p$-almost automorphic
iff for every $y\in Y$ the mapping $f(\cdot,y)$ is $S^{p}$-almost automorphic; that is, for any real sequence $(a_{n})$ there exist a subsequence $(a_{n_{k}})$ of $(a_{n})$ and a map $g : {\mathbb R} \times Y \rightarrow X$ such that $g(\cdot,y)\in L_{loc}^{p}({\mathbb R}:X)$ for all $y\in Y$ as well as that:
\begin{align*}
\lim_{k\rightarrow \infty}\int^{1}_{0}\Bigl \| f\bigl(t+a_{n_{k}}+s,y\bigr) -g(t+s,y)\Bigr \|^{p} \, ds =0
\end{align*}
and
\begin{align*}
\lim_{k\rightarrow \infty}\int^{1}_{0}\Bigl \| g\bigl( t+s-a_{n_{k}},y\bigr) -f(t+s,y)\Bigr \|^{p} \, ds =0
\end{align*}
for each $ t\in {\mathbb R}$ and for each $y\in Y.$ We denote by $AAS^{p}({\mathbb R} \times Y : X)$ the vector space consisting of all such functions.
\end{defn}

\begin{defn}\label{marko-rw} \cite{diagana-definicija}
A function $f : {\mathbb R}\times Y\rightarrow X$ is said to be $S^{p}$-pseudo almost automorphic iff for each $y\in Y$ we have that $f(\cdot,y)\in L^{p}_{loc}({\mathbb R} :  X)$
and $f(\cdot,\cdot)$ admits a decomposition
$
f(t,y)=g(t,y)+q(t,y),
$ $t\in {\mathbb R},$ where $g(\cdot,\cdot)$ is $S^{p}$-almost automorphic and 
\begin{align}\label{stepa-auto-durak-twoa}
\lim_{T\rightarrow +\infty}\frac{1}{2T}\int^{T}_{-T}\Biggl[ \int^{t+1}_{t}\|q(s,y)\|^{p}\, ds \Biggr]^{1/p}\, dt =0,
\end{align}
uniformly on bounded subsets of $Y.$
Denote by $S^{p}PAA({\mathbb R} \times Y,X, \rho_{1}, \rho_{2})$
the collection of such functions, and by $S^{p}PAA_{0}({\mathbb R} \times Y,X, \rho_{1}, \rho_{2})$ the collection of functions satisfying that for each $y\in Y$ one has that $q(\cdot)$ is a locally $p$-integrable $X$-valued function  and \eqref{stepa-auto-durak-twoa} holds. If $\rho=\rho_{1}=\rho_{2},$ then we write $S^{p}PAA({\mathbb R} \times Y,X, \rho_{1}, \rho_{2})\equiv S^{p}PAA({\mathbb R} \times Y, X,\rho).$ 
\end{defn}

We refer the reader to \cite{element} for various composition principles for almost automorphic functions, pseudo-almost automorphic functions and Stepanov almost automorphic functions. The reader may consult S. Abass \cite{sead-abas} for the notion of a Weyl pseudo-almost automorphic
function as well as for the formulation of a composition principle for Weyl pseudo-almost automorphic
functions. 

\section{Weighted pseudo-almost periodic functions, weighted pseudo-almost automorphic functions and their generalizations}\label{weighted-primonja}

Set ${\mathbb U}:=\{ \rho \in L_{loc}^{1}({\mathbb R}) : \rho(t)>0\mbox{ a.e. }t\in {\mathbb R}\},$ 
${\mathbb U}_{\infty}:=\{ \rho \in {\mathbb U} : \inf_{x\in {\mathbb R}}\rho (x)<\infty\mbox{ and }\nu(T,\rho):=\lim_{T\rightarrow +\infty}\int^{T}_{-T}\rho(t)\, dt =\infty\}$
and
${\mathbb U}_{b}:=L^{\infty}({\mathbb R}) \cap {\mathbb U}_{\infty}.$
Then it is clear that ${\mathbb U}_{b} \subseteq {\mathbb U}_{\infty} \subseteq {\mathbb U}.$
We say that weights $\rho_{1}(\cdot)$ and $\rho_{2}(\cdot)$ are equivalent, $\rho_{1} \sim \rho_{2}$ for short, iff $\rho_{1}/\rho_{2} \in {\mathbb U}_{b}.$ By ${\mathbb U}_{T}$ we denote the space consisting of all weights $\rho \in {\mathbb U}_{\infty}$ which are equivalent with all its translations.

Next, we introduce the spaces consisting of so-called ergodic components, depending on one or two variables. 

Assume that $\rho_{1},\ \rho_{2}\in {\mathbb U}_{\infty}.$ Set
$$
PAP_{0}\bigl({\mathbb R}, X,\rho_{1},\rho_{2}\bigr):=\Biggl\{ f\in C_{b}({\mathbb R}: X) : \lim_{T\rightarrow +\infty}\frac{1}{2\int^{T}_{-T}\rho_{1}(t)\, dt}\int^{T}_{-T}\|f(t)\| \rho_{2}(t)\, dt =0 \Biggr\}
$$
and
\begin{align*}
& PAP_{0}\bigl({\mathbb R} \times Y, X,\rho_{1},\rho_{2}\bigr):=\Biggl\{ f\in C_{b}({\mathbb R} \times Y: X) :
\\ & \lim_{T\rightarrow +\infty}\frac{1}{2\int^{T}_{-T}\rho_{1}(t)\, dt}\int^{T}_{-T}\|f(t,y)\| \rho_{2}(t)\, dt =0,
\mbox{ uniformly on bounded subsets of }Y\Biggr\}.
\end{align*}

The class of weighted pseudo-almost periodic functions was introduced by T. Diagana in \cite{tokos1} (2006); we will slightly generalize the notion from this paper by examining two weight functions, leading to the concept of so-called double-weighted pseudo almost-periodic functions, as it has been done by T. Diagana in \cite{diagana-double}-\cite{diagana-doublep}.

\begin{defn}\label{novi-weighted}
\begin{itemize}
\item[(i)] A function $f\in C_{b}({\mathbb R}: X)$ is said to be weighted pseudo-almost periodic iff it admits a decomposition
$
f(t)=g(t)+q(t),
$ $t\in {\mathbb R},$ where $g(\cdot)$ is almost periodic and $q(\cdot) \in  PAP_{0}({\mathbb R}, X,\rho_{1}, \rho_{2}).$
\item[(ii)]  A function $f(\cdot,\cdot) \in C_{b}({\mathbb R}\times Y: X)$ is said to be weighted pseudo-almost periodic iff it admits a decomposition
$
f(t,y)=g(t,y)+q(t,y),
$ $t\in {\mathbb R},$ where $g(\cdot,\cdot)$ is almost periodic and $q(\cdot,\cdot) \in  PAP_{0}({\mathbb R} \times Y, X,\rho_{1}, \rho_{2}).$
\end{itemize}
Denote by $WPAP({\mathbb R}, X, \rho_{1}, \rho_{2})$  (respectively, $WPAP({\mathbb R} \times Y,X, \rho_{1}, \rho_{2})$)
the vector spaces of such functions. If $\rho=\rho_{1}=\rho_{2},$ then we write $WPAP({\mathbb R}, X, \rho_{1}, \rho_{2})\equiv WPAP({\mathbb R}, X,\rho)$ and $WPAP({\mathbb R} \times Y,X, \rho_{1}, \rho_{2})\equiv WPAP({\mathbb R} \times Y, X,\rho).$ 
\end{defn}

The class of weighted pseudo-almost automorphic
functions has been introduced by J. Blot, G. M. Mophou, G. M. N'Gu\' er\' ekata and D. Pennequin in \cite{blot-weighted}. The following slight generalization is introduced by T. Diagana \cite{diagana-double}-\cite{diagana-doublep} (see also T. Diagana, K. Ezzinbi, M. Miraoui \cite{diagana-doublepp} and S. Abbas, V. Kavitha, R. Murugesu \cite{abbas-indian}):

\begin{defn}\label{novi-weighted-auto}
\begin{itemize}
\item[(i)] A function $f\in C_{b}({\mathbb R}: X)$ is said to be weighted pseudo-almost automorphic iff it admits a decomposition
$
f(t)=g(t)+q(t),
$ $t\in {\mathbb R},$ where $g(\cdot)$ is almost automorphic and $q(\cdot) \in  PAP_{0}({\mathbb R}, X,\rho_{1}, \rho_{2}).$
\item[(ii)]  A function $f(\cdot,\cdot) \in C_{b}({\mathbb R}\times Y: X)$ is said to be weighted pseudo-almost automorphic iff it admits a decomposition
$
f(t,y)=g(t,y)+q(t,y),
$ $t\in {\mathbb R},$ where $g(\cdot,\cdot)$ is almost automorphic and $q(\cdot,\cdot) \in  PAP_{0}({\mathbb R} \times Y, X,\rho_{1}, \rho_{2}).$
\end{itemize}
Denote by $WPAA({\mathbb R}, X, \rho_{1}, \rho_{2})$  (respectively, $WPAA({\mathbb R} \times Y,X, \rho_{1}, \rho_{2})$)
the vector spaces of such functions. If $\rho=\rho_{1}=\rho_{2},$ then we write $WPAA({\mathbb R}, X, \rho_{1}, \rho_{2})\equiv WPAA({\mathbb R}, X,\rho)$ and $WPAA({\mathbb R} \times Y,X, \rho_{1}, \rho_{2})\equiv WPAA({\mathbb R} \times Y, X,\rho).$ 
\end{defn}

It is well-known that, for every $\rho_{1},\ \rho_{2} \in {\mathbb U}_{T},$ we have that $WPAP({\mathbb R}, X, \rho_{1}, \rho_{2})$ and $WPAA({\mathbb R}, X, \rho_{1}, \rho_{2})$ are Banach spaces with the sup-norm.

For the Stepanov class, we will use the following definition from \cite{abbas-indian} (see also Z. Xia, M. Fan \cite{stepanoff-xia} for the case
that $\rho_{1}=\rho_{2}$):

\begin{defn}\label{novi-weighted-auto-stepa} 
Let $1\leq p<\infty.$
\begin{itemize}
\item[(i)] A Stepanov $p$-bounded function $f(\cdot)$ is said to be weighted Stepanov $p$-pseudo almost periodic (automorphic) iff it admits a decomposition
$
f(t)=g(t)+q(t),
$ $t\in {\mathbb R},$ where $g(\cdot)$ is $S^{p}$-almost periodic (automorphic) and $q(\cdot) \in L_{loc}^{p}({\mathbb R} : X)$ satisfies
\begin{align}\label{stepa-auto-durak-simici}
\lim_{T\rightarrow +\infty}\frac{1}{2\int^{T}_{-T}\rho_{1}(t)\, dt}\int^{T}_{-T}\Biggl[ \int^{t+1}_{t}\|q(s)\|^{p}\, ds \Biggr]^{1/p} \rho_{2}(t)\, dt =0.
\end{align}
Denote by  $S^{p}WPAP({\mathbb R},X, \rho_{1}, \rho_{2})$ ($S^{p}WPAA({\mathbb R},X, \rho_{1}, \rho_{2})$) the vector space of such functions,
and by $S^{p}WPAA_{0}({\mathbb R},X, \rho_{1}, \rho_{2})$ the vector space of locally $p$-integrable $X$-valued functions $q(\cdot)$ such that \eqref{stepa-auto-durak-simici} holds.
\item[(ii)]  A function $f : {\mathbb R}\times Y\rightarrow X$ is said to be weighted $S^{p}$-pseudo almost periodic (automorphic) iff for each $y\in Y$ we have that $f(\cdot,y)\in L^{p}_{loc}({\mathbb R} :  X)$
and $f(\cdot,\cdot)$ admits a decomposition
$
f(t,y)=g(t,y)+q(t,y),
$ $t\in {\mathbb R},$ where $g(\cdot,\cdot)$ is $S^{p}$-almost periodic (automorphic) and 
\begin{align}\label{stepa-auto-durak-two}
\lim_{T\rightarrow +\infty}\frac{1}{2\int^{T}_{-T}\rho_{1}(t)\, dt}\int^{T}_{-T}\Biggl[ \int^{t+1}_{t}\|q(s,y)\|^{p}\, ds \Biggr]^{1/p} \rho_{2}(t)\, dt =0,
\end{align}
uniformly on bounded subsets of $Y.$
\end{itemize}
Denote by $S^{p}WPAP({\mathbb R} \times Y,X, \rho_{1}, \rho_{2})$ ($S^{p}WPAA({\mathbb R} \times Y,X, \rho_{1}, \rho_{2})$)
the vector space of such functions, and by $S^{p}WPAA_{0}({\mathbb R} \times Y,X, \rho_{1}, \rho_{2})$ the vector space of functions satisfying that for each $y\in Y$ one has that $q(\cdot)$ is a locally $p$-integrable $X$-valued function  and \eqref{stepa-auto-durak-two} holds. If $\rho=\rho_{1}=\rho_{2},$ then we write $S^{p}WPAP({\mathbb R} \times Y,X, \rho_{1}, \rho_{2})\equiv S^{p}WPAP({\mathbb R} \times Y, X,\rho)$ ($S^{p}WPAA({\mathbb R} \times Y,X, \rho_{1}, \rho_{2})\equiv S^{p}WPAA({\mathbb R} \times Y, X,\rho)$). 
\end{defn}

For instance, the function $f (t) := \text{sign}(\cos 2\pi \theta t)    +e^{-|t|},$ $t\in {\mathbb R} ,$ where $\theta$ is an irrational number,
is weighted $S^{p}$-pseudo almost
automorphic with the weight functions $\rho_{1}(t) := 1 + t^{2},$ $\rho_{2}(t) := 1$ ($t\in {\mathbb R} $); this function is also $S^{p}$-pseudo almost automorphic (cf. \cite[Example 1]{abbas-indian}).

Denote by ${\mathbb V}_{\infty}$ the collection of all weighted $\rho_{1},\ \rho_{2}\in {\mathbb U}_{\infty}$ such that $$
\limsup_{T\rightarrow +\infty}\frac{\rho_{2}(t+\tau)}{\rho_{2}(t)} <\infty\mbox{ for any }\tau \in {\mathbb R}, \mbox{ and }\limsup_{T\rightarrow +\infty}\frac{\int^{T}_{-T}\rho_{1}(t)\, dt}{\int^{T}_{-T}\rho_{2}(t)\, dt}<\infty.
$$ 
Then, owing to \cite[Theorem 2.1]{abbas-indian}, we have that $S^{p}WPAA({\mathbb R} \times Y,X, \rho_{1}, \rho_{2})$ is a Banach space endowed with the Stepanov norm $\|\cdot \|_{S^{p}},$ for any $p\in [1,\infty).$ A similar statement holds for the space  $S^{p}WPAP({\mathbb R} \times Y,X, \rho_{1}, \rho_{2}).$

We introduce the notions of a weighted Weyl $p$-pseudo almost periodic (automorphic) function and a weighted Besicovitch $p$-pseudo almost periodic (automorphic) function as follows:

\begin{defn}\label{novi-weighted-auto-stepa-weylbesik} 
Let $1\leq p<\infty.$
\begin{itemize}
\item[(i)] A function $f\in L_{loc}^{p}({\mathbb R} : X)$ is said to be weighted Weyl $p$-pseudo almost periodic (automorphic) iff it admits a decomposition
$
f(t)=g(t)+q(t),
$ $t\in {\mathbb R},$ where $g(\cdot)$ is $W^{p}$-almost periodic (automorphic) and $q(\cdot) \in L_{loc}^{p}({\mathbb R} : X)$ satisfies
\begin{align}\label{stepa-auto-durak-weyl}
\lim_{T\rightarrow +\infty}\frac{1}{2\int^{T}_{-T}\rho_{1}(t)\, dt}\int^{T}_{-T}\Biggl[ \lim_{l\rightarrow +\infty}\frac{1}{2l}\int^{t+l}_{t-l}\|q(s)\|^{p}\, ds \Biggr]^{1/p} \rho_{2}(t)\, dt =0.
\end{align}
Denote by $W^{p}WPAA({\mathbb R},X, \rho_{1}, \rho_{2})$ ($W^{p}WPAA({\mathbb R},X, \rho_{1}, \rho_{2})$) the collection of such functions,
and by $W^{p}WPAA_{0}({\mathbb R},X, \rho_{1}, \rho_{2})$ the collection of locally $p$-integrable $X$-valued functions $q(\cdot)$ such that \eqref{stepa-auto-durak-weyl} holds.
\item[(ii)] A function $f\in L_{loc}^{p}({\mathbb R} : X)$ is said to be weighted Besicovitch $p$-pseudo almost periodic (automorphic) iff it admits a decomposition
$
f(t)=g(t)+q(t),
$ $t\in {\mathbb R},$ where $g(\cdot)$ is $B^{p}$-almost periodic (automorphic) and $q(\cdot) \in L_{loc}^{p}({\mathbb R} : X)$ satisfies
\begin{align}\label{stepa-auto-durak-besik}
\lim_{T\rightarrow +\infty}\frac{1}{2\int^{T}_{-T}\rho_{1}(t)\, dt}\int^{T}_{-T}\Biggl[ \limsup_{l\rightarrow +\infty}\frac{1}{2l}\int^{t+l}_{t-l}\|q(s)\|^{p}\, ds \Biggr]^{1/p} \rho_{2}(t)\, dt =0.
\end{align}
Denote by $B^{p}WPAP({\mathbb R},X, \rho_{1}, \rho_{2})$ ($B^{p}WPAA({\mathbb R},X, \rho_{1}, \rho_{2})$) the collection of such functions,
and by $B^{p}WPAA_{0}({\mathbb R},X, \rho_{1}, \rho_{2})$ the collection of locally $p$-integrable $X$-valued functions $q(\cdot)$ such that \eqref{stepa-auto-durak-besik} holds.
\end{itemize}
\end{defn}

We will not introduce here the notions of two-parameter weighted Weyl $p$-pseudo almost periodic (automorphic) functions and two-parameter weighted Besicovitch $p$-pseudo almost periodic (automorphic) functions since these notions can be very difficultly applied in the analysis of abstract semilinear Cauchy inclusions. It is easily seen that $B^{p}WPAA({\mathbb R},X, \rho_{1}, \rho_{2})$ is a vector space, as well as that $W^{p}WPAA({\mathbb R},X, \rho_{1}, \rho_{2})\subseteq B^{p}WPAA({\mathbb R},X, \rho_{1}, \rho_{2}).$ Similar statements hold in the case of consideration of weighted pseudo-almost periodicity.

Concerning the convolution invariance of the space $B^{p}WPAA_{0}({\mathbb R},X, \rho_{1}, \rho_{2}),$ we have the following result with $p=1$:

\begin{prop}\label{setnja}
Let  $q(\cdot) \in L_{loc}^{1}({\mathbb R} : X)$ be $S^{1}$-bounded,  let $g\in L^{1}({\mathbb R}),$ and let $q\in B^{1}WPAA_{0}({\mathbb R},X, \rho_{1}, \rho_{2}).$
Then $g\ast q\in  B^{1}WPAA_{0}({\mathbb R},X, \rho_{1}, \rho_{2}).$
\end{prop}

\begin{proof}
First of all, let us recall the well-known fact that
\begin{align}\label{sm-aumorfic}
\limsup_{l\rightarrow +\infty}\frac{1}{2l}\int^{t+l}_{t-l}\|q(s)\|\, ds =\limsup_{l\rightarrow +\infty}\frac{1}{2l}\int^{t+l-r}_{t-l-r}\|q(s)\|\, ds,\quad r\in {\mathbb R}. 
\end{align}
Denote 
$$
f_{l,t}(r):=\frac{|g(r)|}{2l}\int^{t+l-r}_{t-l-r}\|q(s)\|\, ds,\quad l\geq 1,\ \ t,\ r\in {\mathbb R}. 
$$
By the $S^{1}$-boundedness of $q(\cdot),$ we get the existence of a function $G_{l,t}(\cdot)\in L^{1}({\mathbb R})$ such that $f_{l,t}(r)\leq G_{l,t}(r)$ for a.e. $r\in {\mathbb R}.$ Hence, we can apply the reverse Fatou's lemma in order to see that
\begin{align}
\notag \limsup_{l\rightarrow +\infty}\int^{\infty}_{-\infty}& \Biggl[\frac{1}{2l}\int^{t+l-r}_{t-l-r}\|q(s)\|\, ds \Biggr] |g(r)|\, dr 
\\ & \label{smor-aumorfic} \leq \int^{\infty}_{-\infty}\limsup_{l\rightarrow +\infty}\Biggl[\frac{1}{2l}\int^{t+l-r}_{t-l-r}\|q(s)\|\, ds\Biggr] |g(r)|\, dr.
\end{align}   
Then the final conclusion follows by using \eqref{sm-aumorfic}-\eqref{smor-aumorfic}, the validity of \eqref{stepa-auto-durak-besik} for $q(\cdot),$ and the following integral computation with Fubini theorem:
\begin{align*}
& \frac{1}{2\int^{T}_{-T}\rho_{1}(t)\, dt}\int^{T}_{-T}\Biggl[ \limsup_{l\rightarrow +\infty}\frac{1}{2l}\int^{t+l}_{t-l}\Biggl\|\int^{\infty}_{-\infty}g(s-r)q(r)\, dr\Biggr\|  ds \Biggr] \rho_{2}(t)\, dt 
\\ & \leq \frac{1}{2\int^{T}_{-T}\rho_{1}(t)\, dt}\int^{T}_{-T}\Biggl[ \limsup_{l\rightarrow +\infty}\frac{1}{2l}\int^{t+l}_{t-l}\int^{\infty}_{-\infty}|g(r)| \|q(s-r)\| \, dr\,  ds \Biggr] \rho_{2}(t)\, dt 
\\ & =\frac{1}{2\int^{T}_{-T}\rho_{1}(t)\, dt}\int^{T}_{-T}\Biggl[ \limsup_{l\rightarrow +\infty}\int^{\infty}_{-\infty}\frac{1}{2l}\int^{t+l}_{t-l}|g(r)| \|q(s-r)\| \, ds\,  dr \Biggr] \rho_{2}(t)\, dt 
\\ & =\frac{1}{2\int^{T}_{-T}\rho_{1}(t)\, dt}\int^{T}_{-T}\Biggl[ \limsup_{l\rightarrow +\infty}\int^{\infty}_{-\infty}\frac{1}{2l}\int^{t+l-r}_{t-l-r} \|q(s)\| \, ds\,  |g(r)|\, dr \Biggr] \rho_{2}(t)\, dt
\\ & \leq \frac{1}{2\int^{T}_{-T}\rho_{1}(t)\, dt}\int^{T}_{-T}\int^{\infty}_{-\infty}\limsup_{l\rightarrow +\infty}\Biggl[\frac{1}{2l}\int^{t+l-r}_{t-l-r}\|q(s)\|\, ds\Biggr] |g(r)|\, dr\, dt
\\ & =\frac{1}{2\int^{T}_{-T}\rho_{1}(t)\, dt}\int^{T}_{-T}\int^{\infty}_{-\infty}\limsup_{l\rightarrow +\infty}\Biggl[\frac{1}{2l}\int^{t+l}_{t-l}\|q(s)\|\, ds\Biggr] |g(r)|\, dr\, dt
\\ & =\frac{\int^{\infty}_{-\infty}|g(r)|\, dr}{2\int^{T}_{-T}\rho_{1}(t)\, dt}\int^{T}_{-T} \Biggl[ \limsup_{l\rightarrow +\infty}\frac{1}{2l}\int^{t+l}_{t-l}\|q(s)\|\, ds \Biggr]  \rho_{2}(t)\, dt .
\end{align*}
\end{proof}

The theory of weighted pseudo-almost periodic functions and weighted pseudo-almost automorphic functions is full of open problems. For example, let $\rho=\rho_{1}=\rho_{2}.$ Then it is well known that decomposition of a weighted pseudo-almost periodic function into its almost-periodic and ergodic component is not unique, in general.
The second problem appearing is that the sum $AP({\mathbb R} : X)$ and $PAP_{0}({\mathbb R}, X,\rho)$ need not be a closed subspace of $C_{b}({\mathbb R}: X)$ albeit the both parts $AP({\mathbb R} : X)$ and $PAP_{0}({\mathbb R}, X,\rho)$ considered separately
are closed subspaces of $C_{b}({\mathbb R}: X).$ To overcome this problem, and to analyze weighted pseudo almost-periodic properties of certain classes of semilinear first order Cauchy problems, J. Zhang, T.-J. Xiao and J. Liang \cite{comp-weighted} have introduced the following norm on the space $PAP({\mathbb R}, X,\rho):$
$$
\|f\|_{\rho}:=\inf_{i\in I}\Biggl[  \sup_{t\in {\mathbb R}}\|g_{i}(t)\|+\sup_{t\in {\mathbb R}}\|q_{i}(t)\|\Biggr],
$$
where $I$ denotes the family of all possible decompositions of $f(\cdot)$ into almost-periodic and ergodic component. This norm turns $PAP({\mathbb R}, X,\rho)$ into a Banach space. The convolution and translation invariance of (double-)weighted pseudo almost-periodic functions and some other problems for this class have been investigated by T. Diagana \cite{diagana-double}-\cite{diagana-doublep},
D. Ji, Ch. Zhang \cite{dji-weighted} and A. Coronel, M. Pinto and D. Sep\' ulveda  \cite{coronel}. The translation invariance of $PAP({\mathbb R}, X,\rho)$ is ensured e.g. by the validity of condition 
\begin{align*}
\sup_{r>0}\sup_{t\in \Omega_{r,s}}\frac{\rho(t+s)}{\rho(t)}<\infty,\quad s\in {\mathbb R};
\end{align*}
cf. \cite[Theorem 3.7 (b)]{coronel}. 

In the following proposition, we provide a slightly different condition ensuring the translation invariance of the space $PAP_{0}({\mathbb R}, X,\rho_{1},\rho_{2}).$ 

\begin{prop}\label{tra-inv-automorphic}
Let $\rho_{1},\ \rho_{2}\in {\mathbb U}_{\infty}.$ Then the space $PAP_{0}({\mathbb R}, X,\rho_{1},\rho_{2})$ is translation invariant if 
\begin{align}\label{this-blues}
\lim_{T \rightarrow +\infty}\frac{\Bigl| \int^{-T}_{-T-s}\rho_{2}(t)\, dt \Bigr| +\Bigl| \int^{T-s}_{T}\rho_{2}(t)\, dt \Bigr|}{\int^{T}_{-T}\rho_{1}(t)\, dt}=0,\quad s\in {\mathbb R}
\end{align}
and there exists a function $g : {\mathbb R} \rightarrow (0,\infty)$ such that 
\begin{align}\label{this-blues1}
\rho_{2}(t-s) \leq g(s)\rho_{2}(t),\quad t,\ s\in {\mathbb R}.
\end{align}
\end{prop}

\begin{proof}
Let $s\in {\mathbb R}$ and $f\in PAP_{0}({\mathbb R}, X,\rho_{1},\rho_{2}).$ We need to prove that $f(-s+\cdot) \in PAP_{0}({\mathbb R}, X,\rho_{1},\rho_{2}).$ Towards this end, observe that \eqref{this-blues1} and a simple calculation yield that for each $T>0$ we have:
\begin{align*}
&\frac{1}{2\int^{T}_{-T} \rho_{1}(t)\, dt} \int^{T}_{-T}\|f(-s+t)\| \rho_{2}(t)\, dt 
\\ & =\frac{1}{2\int^{T}_{-T} \rho_{1}(t)\, dt}\int^{T-s}_{-T-s}\|f(t)\| \rho_{2}(t-s)\, dt 
\\ & \leq \frac{ g(s)}{2\int^{T}_{-T} \rho_{1}(t)\, dt}\int^{T-s}_{-T-s}\|f(t)\| \rho_{2}(t)\, dt 
\\ & \leq \frac{ g(s)}{2\int^{T}_{-T} \rho_{1}(t)\, dt}\Biggl[\int^{T}_{-T}\|f(t)\| \rho_{2}(t)\, dt  +\Biggl| \int^{-T}_{-T-s}\rho_{2}(t)\, dt \Biggr| +\Biggl| \int^{T-s}_{T}\rho_{2}(t)\, dt \Biggr| \Biggr].
\end{align*}
Then the final conclusion follows by applying \eqref{this-blues}.
\end{proof}

In order to analyze generalized weighted pseudo-almost automorphic solutions of semilinear (fractional) Cauchy inclusions, we need to repeat some known facts about composition principles for the classes of weighted pseudo-almost automorphic functions and Stepanov weighted pseudo-almost automorphic functions (weighted pseudo-almost periodic solutions can be analyzed in a similar fashion; see e.g. \cite[Theorem 5.9]{diagana-doublep} and \cite[Theorem 3.1, Theorem 3.5]{comp-weighted} .

\subsection{Composition principles for weighted pseudo-almost automorphic solutions}\label{PIVOTI}

Our first result has been proved by T. Diagana (cf. \cite[Theorem 5.8]{diagana-doublep}):

\begin{thm}\label{diagana-composition}
Assume that $\rho_{1},\ \rho_{2}\in {\mathbb U}_{\infty},$ $f : {\mathbb R}\times Y\rightarrow X$ is weighted pseudo-almost  automorphic, 
and $h : {\mathbb R}\rightarrow Y$ is weighted pseudo-almost automorphic. Assume that there exists a finite constant $L_{f}>0$ such that
\begin{align}\label{minorcas}
\|f(t,y)-f(t,z)\| \leq L_{f}\|y-z\|_{Y},\quad t\in {\mathbb R},\ \  y,\ z\in Y .
\end{align}
Then $f(\cdot,h(\cdot)) \in WPAA({\mathbb R} \times Y,X, \rho_{1}, \rho_{2})$.
\end{thm}

Arguing as in the proof of \cite[Theorem 3.6, Theorem 3.7]{stepanoff-xia}, we can prove the following composition principles, stated here with two generally different pivot spaces (see also \cite[Theorem 2.2, Theorem 2.3, Theorem 2.4]{abbas-indian}).

\begin{thm}\label{xia-composition}
Assume that $\rho_{1},\ \rho_{2}\in {\mathbb U}_{\infty},$ $1\leq p<\infty,$ $f : {\mathbb R}\times Y\rightarrow X$ is weighted $S^{p}$-pseudo almost automorphic, 
$
f(t,y)=g(t,y)+q(t,y),
$ $t\in {\mathbb R},$ where $g(\cdot,\cdot)$ is $S^{p}$-almost automorphic and $q(\cdot,\cdot)$ satisfies \eqref{stepa-auto-durak-two},
uniformly on bounded subsets of $Y.$ Let $h : {\mathbb R}\rightarrow Y$ be weighted $S^{p}$-pseudo almost automorphic, $
f(t)=g(t)+q(t),
$ $t\in {\mathbb R},$ where $g(\cdot)$ is $S^{p}$-almost automorphic with relatively compact range in $Y$, and $q(\cdot) \in L_{loc}^{p}({\mathbb R} : X)$ satisfies
\eqref{stepa-auto-durak-simici}. Assume that there exist two finite constants $L_{f}>0$ and $L_{g}>0$ such that
\eqref{minorcas}
and
\begin{align}\label{minorcas-ibiza}
\|g(t,y)-g(t,z)\| \leq L_{g}\|y-z\|_{Y},\quad t\in {\mathbb R} ,\ \  y,\ z\in Y
\end{align}
hold.
Then $f(\cdot,h(\cdot)) \in S^{p}WPAA({\mathbb R} \times Y,X, \rho_{1}, \rho_{2}).$
\end{thm}

\begin{thm}\label{bibl-auto}
Assume that $\rho_{1},\ \rho_{2}\in {\mathbb U}_{\infty},$ $1< p<\infty,$ $f : {\mathbb R}\times Y\rightarrow X$ is weighted $S^{p}$-pseudo almost automorphic, 
$
f(t,y)=g(t,y)+q(t,y),
$ $t\in {\mathbb R},$ where $g(\cdot,\cdot)$ is $S^{p}$-almost automorphic and $q(\cdot,\cdot)$ satisfies \eqref{stepa-auto-durak-two},
uniformly on bounded subsets of $Y.$ Let $h : {\mathbb R}\rightarrow Y$ be weighted $S^{p}$-pseudo almost automorphic, $
f(t)=g(t)+q(t),
$ $t\in {\mathbb R},$ where $g(\cdot)$ is $S^{p}$-almost automorphic with relatively compact range in $Y$, and $q(\cdot) \in L_{loc}^{p}({\mathbb R} : X)$ satisfies
\eqref{stepa-auto-durak-simici}. Assume that  $ r\geq \max (p, p/p -1)$ and there exist two Stepanov $r$-almost automorphic scalar-valued functions $L_{f}(\cdot)$ and $L_{g}(\cdot)$ such that
\begin{align}\label{minorca}
\|f(t,y)-f(t,z)\| \leq L_{f}(t)\|y-z\|_{Y},\quad t\in {\mathbb R},\ \  y,\ z\in Y ,
\end{align}
and
\begin{align}\label{minorca-ibiza}
\|g(t,y)-g(t,z)\| \leq L_{g}(t)\|y-z\|_{Y},\quad t\in {\mathbb R} ,\ \  y,\ z\in Y.
\end{align}
Set $q:=pr/p+r.$ Then $q\in [1, p)$ and
$f(\cdot,h(\cdot)) \in S^{q}WPAA({\mathbb R} \times Y,X, \rho_{1}, \rho_{2}).$
\end{thm}

\section{Generalized weighted almost periodic (automorphic) properties of convolution products}\label{konvolucije}

We start this section by stating the following important result, expanding thus our research raised in \cite[Subsection 2.1]{EJDE}.

\begin{prop}\label{ravi-and-auto}
Suppose that $1\leq p <\infty,$ $1/p +1/q=1$
and $(R(t))_{t> 0}\subseteq L(X)$ is a strongly continuous operator family satisfying that $M:=\sum_{k=0}^{\infty}\|R(\cdot)\|_{L^{q}[k,k+1]}<\infty .$ If the space $PAP_{0}({\mathbb R},X, \rho_{1}, \rho_{2})$ is translation invariant (see Proposition \ref{tra-inv-automorphic}) and $f : {\mathbb R} \rightarrow X$ is weighted $S^{p}$-almost periodic, resp. weighted $S^{p}$-almost automorphic, then the function $F(\cdot),$ given by
\begin{align}\label{wer}
F(t):=\int^{t}_{-\infty}R(t-s)f(s)\, ds,\quad t\in {\mathbb R},
\end{align}
is well-defined and belongs to the class 
$$
AP({\mathbb R},X, \rho_{1}, \rho_{2})+S^{p}WPAA_{0}({\mathbb R},X, \rho_{1}, \rho_{2}),
$$ 
resp.,
$$
AA({\mathbb R},X, \rho_{1}, \rho_{2})+S^{p}WPAA_{0}({\mathbb R},X, \rho_{1}, \rho_{2}).
$$ 
\end{prop}

\begin{proof}
We will prove the theorem only for weighted $S^{p}$-almost automorphy.
Let $f(\cdot)=g(\cdot)+q(\cdot),$ where $g(\cdot)$ and $q(\cdot)$ satisfy conditions from Definition \ref{novi-weighted-auto-stepa}(i). 
By \cite[Proposition 5]{element}, and $S^{p}$-almost automorphy of $g(\cdot),$ we have that the function $G(\cdot)$ obtained by replacing $f(\cdot)$ in \eqref{wer} by $g(\cdot),$ is almost automorphic. Define $Q_{k}(t):=\int^{k+1}_{k}R(s)q(t-s)\, ds,$ $t\in {\mathbb R}$ ($k\in {\mathbb N}$). Arguing as in the proof of afore-mentioned proposition, we can prove that $Q_{k}(\cdot)$ is bounded and continuous on ${\mathbb R}$ for all $k\in {\mathbb N},$ as well as that $Q_{k}(\cdot)$ converges uniformly to $Q(\cdot):=\int^{\cdot}_{-\infty}R(\cdot-s)q(s)\, ds.$ Therefore, all we need to prove is that, for any integer $k\in {\mathbb N}$ given in advance, \eqref{stepa-auto-durak-simici} holds with the function $q(\cdot)$ replaced therein by $Q_{k}(\cdot).$ By the
 H\"older inequality and an elementary change of variables in double integral, we have the existence of a positive finite constant $c_{k}>0$ such that:
\begin{align*}
&\frac{1}{2\int^{T}_{-T}\rho_{1}(t)\, dt}\int^{T}_{-T}\Biggl[ \int^{t+1}_{t}\|Q_{k}(s)\|^{p}\, ds \Biggr]^{1/p} \rho_{2}(t)\, dt 
\\ & \leq \frac{\|R(\cdot)\|_{L^{q}[k,k+1]}}{2\int^{T}_{-T}\rho_{1}(t)\, dt}\int^{T}_{-T}\Biggl[ \int^{t+1}_{t}\! \! \int^{k+1}_{k}\|q(s-v)\|^{p}\, dv \, ds\Biggr]^{1/p} \rho_{2}(t)\, dt 
\\ & =\frac{\|R(\cdot)\|_{L^{q}[k,k+1]}}{2\int^{T}_{-T}\rho_{1}(t)\, dt}\int^{T}_{-T}\Biggl[ \int^{t+1}_{t}\! \! \int^{s-(k+1)}_{s-k}\|q(v)\|^{p}\, dv \, ds\Biggr]^{1/p} \rho_{2}(t)\, dt 
\\ & \leq \frac{\|R(\cdot)\|_{L^{q}[k,k+1]}}{2\int^{T}_{-T}\rho_{1}(t)\, dt}\int^{T}_{-T}\Biggl[ \int_{t-(k+1)}^{t-k}\! \! \int^{r+(k+1)}_{t}\|q(r)\|^{p}\, ds \, dr\Biggr]^{1/p} \rho_{2}(t)\, dt 
\\ & + \frac{\|R(\cdot)\|_{L^{q}[k,k+1]}}{2\int^{T}_{-T}\rho_{1}(t)\, dt}\int^{T}_{-T}\Biggl[ \int^{t-k}_{t-(k-1)}\! \! \int^{r-k}_{t+1}\|q(r)\|^{p}\, ds \, dr\Biggr]^{1/p} \rho_{2}(t)\, dt 
\\ & \leq \frac{\|R(\cdot)\|_{L^{q}[k,k+1]}}{2\int^{T}_{-T}\rho_{1}(t)\, dt}\int^{T}_{-T}\Biggl[ \int_{t-(k+1)}^{t-k}\! \! |r+k+1-t| \|q(r)\|^{p} \, dr\Biggr]^{1/p} \rho_{2}(t)\, dt 
\\ & + \frac{\|R(\cdot)\|_{L^{q}[k,k+1]}}{2\int^{T}_{-T}\rho_{1}(t)\, dt}\int^{T}_{-T}\Biggl[ \int^{t-k}_{t-(k-1)}\! \! |t+1-r+k|\|q(r)\|^{p}\, ds \, dr\Biggr]^{1/p} \rho_{2}(t)\, dt 
\\ & \leq \frac{c_{k}\|R(\cdot)\|_{L^{q}[k,k+1]}}{2\int^{T}_{-T}\rho_{1}(t)\, dt}\int^{T}_{-T}\Biggl[ \int_{t-(k+1)}^{t-k} \|q(r)\|^{p} \, dr\Biggr]^{1/p} \rho_{2}(t)\, dt 
\\ & +\frac{c_{k}\|R(\cdot)\|_{L^{q}[k,k+1]}}{2\int^{T}_{-T}\rho_{1}(t)\, dt}\int^{T}_{-T}\Biggl[ \int^{t-k}_{t-(k-1)}\|q(r)\|^{p}\, ds \, dr\Biggr]^{1/p} \rho_{2}(t)\, dt 
\\ & =\frac{c_{k}\|R(\cdot)\|_{L^{q}[k,k+1]}}{2\int^{T}_{-T}\rho_{1}(t)\, dt}\int^{T}_{-T}\Biggl[ \int_{t+)}^{t} \|q(r-(k+1))\|^{p} \, dr\Biggr]^{1/p} \rho_{2}(t)\, dt 
\\ & +\frac{c_{k}\|R(\cdot)\|_{L^{q}[k,k+1]}}{2\int^{T}_{-T}\rho_{1}(t)\, dt}\int^{T}_{-T}\Biggl[ \int^{t+1}_{t}\|q(r-(k-1))\|^{p}\, ds \, dr\Biggr]^{1/p} \rho_{2}(t)\, dt ,\quad T>0.
\end{align*}
Now the final conclusion follows from the fact that \eqref{stepa-auto-durak-simici} holds with the function $q(\cdot)$ replaced therein by $Q_{k}(\cdot)$ and the translation invariance of $PAP_{0}({\mathbb R},X, \rho_{1}, \rho_{2}).$
\end{proof}

Results for weighted Besicovitch almost automorphic functions are rather restrictive; in the following theorem, besides of condition \eqref{jebaci-aumorfija} and almost inevitable condition $p=1,$ we use the esssential boundedness of functions $f(\cdot)$ and $g(\cdot)$ as well as the condition \eqref{mallorca} on weights  
$\rho_{1}(\cdot)$ and $\rho_{2}(\cdot):$

\begin{prop}\label{kaca=autmorphic=weighted}
Suppose that $(R(t))_{t> 0}\subseteq L(X)$ is a strongly continuous operator family satisfying that
\begin{align}\label{jebaci-aumorfija}
\int^{\infty}_{0}(1+t)\|R(t)\|\, dt <\infty.
\end{align}
Let $g\in B^{1}AA({\mathbb R} : X),$ and let $g(\cdot)$ be essentially bounded. Assume, further, that $q\in B^{1}WPAA_{0}({\mathbb R},X, \rho_{1}, \rho_{2})$, $q(\cdot)$ is essentially bounded, and $f(\cdot)=g(\cdot)+q(\cdot)$.
Suppose that
\begin{align}\label{mallorca}
\lim_{T\rightarrow +\infty}\frac{\int^{T}_{-T}\rho_{2}(t)\, dt}{\int^{T}_{-T}\rho_{1}(t)\, dt}=0.
\end{align}
Then the function
$F(\cdot) ,$ given by \eqref{wer},
is bounded and belongs to the class 
$$ 
B^{1}AA({\mathbb R} : X)+B^{1}WPAA_{0}({\mathbb R},X, \rho_{1}, \rho_{2}).
$$
\end{prop}

\begin{proof}
By \cite[Proposition 7]{element}, we have that  the function
$G(\cdot) ,$ given by \eqref{wer} with the function $f(\cdot)$ replaced by $g(\cdot)$ therein, belongs to class $ B^{1}AA({\mathbb R} : X).$ The argumentation contained in the proof of this proposition also shows that the function $F(\cdot)$ is essentially bounded, so that it suffices to show that the function $Q(\cdot),$ obtained by replacing the function $f(\cdot)$ by $q(\cdot)$ in \eqref{wer}, satisfies \eqref{stepa-auto-durak-besik}. Towards this end,
we note that
\begin{align*}
&\frac{1}{2\int^{T}_{-T}\rho_{1}(t)\, dt}\int^{T}_{-T}\Biggl[ \limsup_{l\rightarrow +\infty}\frac{1}{2l}\int^{t+l}_{t-l}\|Q(s)\|\, ds \Biggr] \rho_{2}(t)\, dt 
\\ &=\frac{1}{2\int^{T}_{-T}\rho_{1}(t)\, dt}\int^{T}_{-T}\Biggl[ \lim_{l\rightarrow +\infty}\sup_{y\geq l}\frac{1}{2y}\int^{t+y}_{t-y}\Biggl\|\int^{\infty}_{0}R(v)q(s-v)\, dv\Biggr\|\, ds\Biggr] \rho_{2}(t)\, dt 
\\ & = \frac{1}{2\int^{T}_{-T}\rho_{1}(t)\, dt}\int^{T}_{-T}\Biggl[\lim_{l\rightarrow +\infty}\sup_{y\geq l}\frac{1}{2y}\int^{\infty}_{0}\int^{t+y}_{t-y}\|R(v)\| \|q(s-v)\|\, ds\, dv \Biggr] \rho_{2}(t)\, dt 
\\ & = \frac{1}{2\int^{T}_{-T}\rho_{1}(t)\, dt}\int^{T}_{-T}\Biggl[\lim_{l\rightarrow +\infty}\sup_{y\geq l}\frac{1}{2y}\int^{\infty}_{0}\|R(v)\| \int^{t+y-v}_{t-y-v}\|q(r)\|\, dr \, dv \Biggr] \rho_{2}(t)\, dt 
\\ & \leq \frac{\|q\|_{\infty}}{2\int^{T}_{-T}\rho_{1}(t)\, dt}\int^{T}_{-T}\Biggl[\lim_{l\rightarrow +\infty}\sup_{y\geq l}\int^{\infty}_{0}\|R(v)\| \, dv \Biggr] \rho_{2}(t)\, dt 
\\ &=\frac{\|q\|_{\infty}\int^{\infty}_{0}\|R(v)\| \, dv}{2\int^{T}_{-T}\rho_{1}(t)\, dt}\int^{T}_{-T}\rho_{2}(t)\, dt ,\quad T>0.
\end{align*}
The proof of proposition follows by applying \eqref{mallorca}.
\end{proof}

\begin{rem}\label{gejak}
Before proceeding further, let us only observe that the validity of \eqref{mallorca} implies that for any $q\in B^{1}WPAA_{0}({\mathbb R},X, \rho_{1}, \rho_{2})$ such that $q(\cdot)$ is essentially bounded, we have that the condition $\int^{\infty}_{0}\|R(v)\| \, dv <\infty$
is sufficient to ensure that the function $Q(\cdot)$ defined above is in class $B^{1}WPAA_{0}({\mathbb R},X, \rho_{1}, \rho_{2}).$ Furthermore, the condition $\int^{\infty}_{0}\|R(v)\| \, dv <\infty$ implies that the function $G(\cdot)$ defined above is almost automorphic provided that $g(\cdot)$ is almost automorphic (see \cite[Proposition 4]{element}).
\end{rem}

\begin{rem}\label{gejak-dzibut}
As already mentioned, the class of Besicovitch-Doss $p$-almost periodic functions has been recently introduced in \cite{NSJOM-besik}, $1\leq p<\infty.$ Denote by
${\mathrm B}^{1} AP({\mathbb R} : X)$ the corresponding class with $p=1.$ Then the validity of \eqref{jebaci-aumorfija} implies that
the function $G(\cdot),$ obtained by replacing $f(\cdot)$ in \eqref{wer} by $g(\cdot),$ is again in the class $ 
{\mathrm B}^{1}AA({\mathbb R} : X);$ cf. \cite[Theorem 3.1]{NSJOM-besik}.
\end{rem}

For our investigations of the finite convolution product, we need to slightly addapt the notation used so far. Set ${\mathbb U}_{p}:=\{ \rho \in L_{loc}^{1}([0,\infty)) : \rho(t)>0\mbox{ a.e. }t\geq 0\},$ ${\mathbb U}_{b,p}:=\{ \rho \in L^{\infty}([0,\infty)) : \rho(t)>0\mbox{ a.e. }t\geq 0\}$ and
${\mathbb U}_{\infty, p}:=\{ \rho \in {\mathbb U} : \nu(T,\rho):=\lim_{T\rightarrow +\infty}\int^{T}_{0}\rho(t)\, dt =\infty\}.$ Then ${\mathbb U}_{b,p} \subseteq {\mathbb U}_{\infty, p} \subseteq {\mathbb U}_{p}.$
If $\rho_{1},\ \rho_{2}\in {\mathbb U}_{\infty ,p},$ then we set
\begin{align*}
PAP_{0}& \bigl([0,\infty), X,\rho_{1},\rho_{2}\bigr)
\\ & :=\Biggl\{ f\in C_{b}([0,\infty): X) : \lim_{T\rightarrow +\infty}\frac{1}{\int^{T}_{0}\rho_{1}(t)\, dt}\int^{T}_{0}\|f(t)\| \rho_{2}(t)\, dt =0 \Biggr\}
\end{align*}
and
\begin{align*}
& PAP_{0}\bigl([0,\infty) \times Y, X,\rho_{1},\rho_{2}\bigr):=\Biggl\{ f\in C_{b}([0,\infty) \times Y: X) :
\\ & \lim_{T\rightarrow +\infty}\frac{1}{\int^{T}_{0}\rho_{1}(t)\, dt}\int^{T}_{0}\|f(t,y)\| \rho_{2}(t)\, dt =0,
\mbox{ uniformly on bounded subsets of }Y\Biggr\};
\end{align*}
see also \cite{blot} and \cite{diagana-doublepp}.

Concerning the invariance of space $PAP_{0}([0,\infty), X,\rho_{1},\rho_{2})$ under the action of finite convolution product,
we have the following result:

\begin{prop}\label{nek-nea}
Assume that there exists a non-negative measurable function $g : [0,\infty) \rightarrow [0,\infty)$ such that $\rho_{2}(t)\leq g(s)\rho_{2}(t-s)$ for $0\leq s\leq t<\infty.$ Assume that $(R(t))_{t>0}\subseteq L(X)$ is a strongly continuous operator family satisfying
\begin{align}\label{milica-sekica}
\int^{\infty}_{0}\bigl(1+g(s)\bigr)\|R(s)\|\, ds <\infty.
\end{align}
Let $f \in PAP_{0}([0,\infty), X,\rho_{1},\rho_{2}).$ Define
$
F(t):=\int^{t}_{0}R(t-s)f(s)\, ds,$ $t\geq 0.
$
Then we have $F \in PAP_{0}([0,\infty), X,\rho_{1},\rho_{2}).$
\end{prop}

\begin{proof}
It can be easily verified that $F\in C_{b}([0,\infty) : X);$ cf. \cite[Chapter 1]{a43} for more details. The claimed statement follows from the prescribed assumptions and the next computation:
\begin{align*}
\frac{1}{\int^{T}_{0} \rho_{1}(t)\, dt}& \int^{T}_{0}\|F(t)\| \rho_{2}(t)\, dt
\\ & \leq \frac{1}{\int^{T}_{0}\rho_{1}(t)\, dt}\int^{T}_{0}\Biggl[\int^{t}_{0}\|R(s)\|  \|f(t-s)\|\, ds \Biggr]  \rho_{2}(t)\, dt
\\ & =\frac{1}{\int^{T}_{0}\rho_{1}(t)\, dt}\int^{T}_{0}\int^{T}_{s}\|R(s)\|  \|f(t-s)\| \rho_{2}(t) \, dt \, ds
\\ & \leq \frac{1}{\int^{T}_{0}\rho_{1}(t)\, dt}\int^{T}_{0}g(s) \|R(s)\|\Biggl[ \int^{T}_{s}  \|f(t-s)\| \rho_{2}(t-s) \, dt\Biggr] \, ds
\\ & =\frac{1}{\int^{T}_{0}\rho_{1}(t)\, dt}\int^{T}_{0}g(s) \|R(s)\|\Biggl[ \int^{T-s}_{0}  \|f(r)\| \rho_{2}(r) \, dr\Biggr] \, ds
\\ & \leq \Biggl[ \int^{\infty}_{0}g(s)\|R(s)\|\, ds \Biggr] \cdot \Biggl[ \frac{1}{\int^{T}_{0}\rho_{1}(t)\, dt} \int^{T}_{0}  \|f(r)\| \rho_{2}(r) \, dr\Biggr],\quad T>0.
\end{align*} 
\end{proof}

\begin{rem}\label{velinov-auto}
Assume, in place of condition \eqref{milica-sekica}, that $1\leq p<\infty,$ $1/p+1/q=1,$  
$M=\sum_{k=0}^{\infty}\|R(\cdot)\|_{L^{q}[k,k+1]}<\infty $ and $g(\cdot)$ additionally satisfies that it is $S^{p}$-bounded. Then, again, $F \in PAP_{0}([0,\infty), X,\rho_{1},\rho_{2})$ and here it is only worth noting that we can use similar arguments as above and the following estimate
$$
\int^{T}_{0}g(s)\|R(s)\|\, ds \leq \sum_{k=0}^{\lceil T \rceil}\|R(\cdot)\|_{L^{q}[k,k+1]} \|g(\cdot)\|_{L^{p}[k,k+1]} \leq M \|g\|_{S^{p}},\quad T>0. 
$$
\end{rem}

Combining Proposition \ref{ravi-and-auto}, Proposition \ref{nek-nea}-Remark \ref{velinov-auto} and the argumentation contained in the proof of \cite[Proposition 2.13]{EJDE}, we can clarify the following proposition (weighted $S^{p}$-almost periodic case can be considered similarly):

\begin{prop}\label{Cauchy}
Assume that $1\leq p<\infty,$ $1/p+1/q=1,$ and there exists a $S^{p}$-bounded function $g : [0,\infty) \rightarrow [0,\infty)$ such that $\rho_{2}(t)\leq g(s)\rho_{2}(t-s)$ for $0\leq s\leq t<\infty.$ Assume, further, that $(R(t))_{t>0}\subseteq L(X)$ is a strongly continuous operator family satisfying that for each $s\geq 0$ we have $M_{s}:=\sum_{k=0}^{\infty}\|R(\cdot)\|_{L^{q}[s+k,s+k+1]}<\infty ,$ as well as that the space $PAA_{0}({\mathbb R},X, \rho_{1}, \rho_{2})$ is translation invariant and $g : {\mathbb R} \rightarrow X$ is weighted $S^{p}$-almost automorphic. If  $q \in PAP_{0}([0,\infty), X,\rho_{1},\rho_{2}),$ then the function ${\mathbf F}(\cdot),$ given by
\begin{align}\label{wer}
{\mathbf F}(t):=\int^{\infty}_{0}R(t-s)\bigl[g(s)+q(s)\bigr]\, ds,\quad t\geq 0,
\end{align}
is well-defined and belongs to the class 
\begin{align*}
AA_{[0,\infty)}({\mathbb R},X, \rho_{1}, \rho_{2})&+S^{p}WPAA_{0}^{[0,\infty)}({\mathbb R},X, \rho_{1}, \rho_{2})
\\ & +S^{p}_{0}([0,\infty) :X)+PAP_{0}([0,\infty), X,\rho_{1},\rho_{2}).
\end{align*}
Here, $AA_{[0,\infty)}({\mathbb R},X, \rho_{1}, \rho_{2})$ and $S^{p}WPAA_{0}^{[0,\infty)}({\mathbb R},X, \rho_{1}, \rho_{2})$
denote the spaces consisting of restrictions of functions belonging to $AA({\mathbb R},X, \rho_{1}, \rho_{2})$ and\\ $S^{p}WPAA_{0}({\mathbb R},X, \rho_{1}, \rho_{2})$ to the non-negative real axis, respectively.
\end{prop} 

\section{Weighted pseudo-almost automorphic solutions of semilinear (fractional) Cauchy inclusions}\label{semilinear-spring}

In this section, we will clarify a few results concerning the existence and uniqueness of weighted automorphic solutions of semilinear (fractional) Cauchy inclusions; the existence and uniqueness of weighted automorphic solutions can be analyzed similarly.  
Because of some obvious complications appearing in the study of existence and uniqueness of weighted pseudo-almost periodic (automorphic) solutions defined only for non-negative values of time $t$ (see Proposition \ref{Cauchy}), henceforth we will restrict ourselves to the study of abstract Cauchy inclusions \eqref{favini}-\eqref{left}, only.
The organization of section is very similar to that of \cite[Section 5]{element}; the proofs of structural results are omitted since they can be deduced by using composition principles from Subsection \ref{PIVOTI} and the argumentation contained in \cite{nova-mono}.

We deal with the class of
multivalued linear operators ${\mathcal A}$ satisfying the condition \cite[(P), p. 47]{faviniyagi} introduced by A. Favini and A. Yagi:
\begin{itemize} \index{removable singularity at zero}
\item[(P)]
There exist finite constants $c,\ M>0$ and $\beta \in (0,1]$ such that\index{condition!(PW)}
$$
\Psi:=\Psi_{c}:=\Bigl\{ \lambda \in {\mathbb C} : \Re \lambda \geq -c\bigl( |\Im \lambda| +1 \bigr) \Bigr\} \subseteq \rho({\mathcal A})
$$
and
$$
\| R(\lambda : {\mathcal A})\| \leq M\bigl( 1+|\lambda|\bigr)^{-\beta},\quad \lambda \in \Psi .
$$
\end{itemize}
We define the fractional power $(-{\mathcal A})^{\theta}$ for $\theta >\beta-1$ as usually.
Set $Y:=[D((-{\mathcal A})^{\theta})]$ and $\|\cdot\|_{Y}:=\|\cdot\|_{[D((-{\mathcal A})^{\theta})]};$ then $Y$ is a Banach space that is continuously embedded in $X.$
Define, further,
\begin{align*}
T_{\nu}(t)x:=\frac{1}{2\pi i}\int_{\Gamma}(-\lambda)^{\nu}e^{\lambda t}\bigl(  \lambda -{\mathcal A} \bigr)^{-1}x\, d\lambda,\quad x\in X,\ t>0 \ (\nu>0),
\end{align*}
where $\Gamma$ denotes the upwards oriented curve $\lambda=-c(|\eta|+1)+i\eta$ ($\eta \in {\mathbb R}$).
Then there exists a finite constant $M>0$ such that:
\begin{itemize}
\item[(A)] $\|  T_{\nu}(t) \|\leq Me^{-ct}t^{ \beta-\nu-1},\ t>0,\ \nu>0.$
\end{itemize}
Assume that $L_{f}(\cdot)$ is a locally bounded non-negative function, and $M$ denotes the constant from (A), with $\nu=\theta.$
Set, for every $n\in {\mathbb N},$
\begin{align}
\notag M_{n}:=M^{n}\sup_{t\in {\mathbb R}}& \int^{t}_{-\infty}\int^{x_{n}}_{-\infty}\cdot \cdot \cdot \int^{x_{2}}_{-\infty}e^{-c(t-x_{n})}
\bigl(t-x_{n}\bigr)^{\beta -\theta-1}
\\\label{em-en-a} & \times  \prod^{n}_{i=2}e^{-c(x_{i}-x_{i-1})}\bigl(x_{i}-x_{i-1}\bigr)^{\beta -\theta-1} \prod^{n}_{i=1}L_{f}(x_{i})\, dx_{1}\, dx_{2}\cdot \cdot \cdot \, dx_{n}.
\end{align}

Define
\begin{align*}
T_{\gamma,\nu}(t)x:=t^{\gamma \nu}\int^{\infty}_{0}s^{\nu}\Phi_{\gamma}( s)T_{0}\bigl( st^{\gamma}\bigr)x\, ds,\quad t>0,\ x\in X,\ \nu >-\beta,
\end{align*}
$$
S_{\gamma}(t):=T_{\gamma,0}(t)\mbox{ and }P_{\gamma}(t):=\gamma T_{\gamma,1}(t)/t^{\gamma},\quad t>0,
$$
where 
$\Phi_{\gamma}(\cdot)$ denotes the famous Wright function\index{function!Wright}, defined by
$$
\Phi_{\gamma}(z):=\sum \limits_{n=0}^{\infty}
\frac{(-z)^{n}}{n! \Gamma (1-\gamma -\gamma n)},\quad z\in {\mathbb C}.
$$
Set also
\begin{align*}
R_{\gamma}(t):=& t^{\gamma -1}P_{\gamma}(t),\ t>0\mbox{ and }
\\ & R_{\gamma}^{\theta}(t):=\gamma t^{\gamma -1}\int^{\infty}_{0}s\Phi_{\gamma}(s)T_{\theta}\bigl(st^{\gamma}\bigr)x\, ds,\ t>0,\ x\in X.
\end{align*}

Consider the condition \eqref{minorcas}
with $L_{f}(\cdot)$ being a measurable non-negative function. Set, for every $n\in {\mathbb N},$
\begin{align*}
B_{n}:=&\sup_{t\geq 0}\int^{t}_{-\infty}\int^{x_{n}}_{-\infty}\cdot \cdot \cdot \int^{x_{2}}_{-\infty}
\bigl\|R_{\gamma}^{\theta}(t-x_{n})\bigr\|
\\ & \times \prod^{n}_{i=2}\bigl\|R_{\gamma}^{\theta}(x_{i}-x_{i-1})\bigr\| \prod^{n}_{i=1}L_{f}(x_{i})\, dx_{1}\, dx_{2}\cdot \cdot \cdot \, dx_{n}.
\end{align*}

If $(Z,\|\cdot \|_{Z})$ is a complex Banach space that is continuously embedded in $X,$ then we use the following notion of a mild solution of \eqref{favini}, resp., \eqref{left}:

\begin{defn}\label{joka}
Assume that $f : I \times Z \rightarrow X.$
By a mild solution of \eqref{favini}, we mean any $ Z$-continuous function $u(\cdot)$ such that $u(t)= (\Lambda u)(t),$ $t\in {\mathbb R},$ where
$$
t\mapsto (\Lambda u)(t):=\int_{-\infty}^{t}T(t-s)f(s,u(s))\, ds,\ t\in {\mathbb R}.
$$
\end{defn}

\begin{defn}\label{durak}
Assume that $f : I \times Z \rightarrow X.$
By a mild solution of \eqref{left}, we mean any $ Z$-continuous function $u(\cdot)$ such that $u(t)= (\Lambda_{\gamma} u)(t),$ $t\in {\mathbb R},$ where
$$
t\mapsto (\Lambda_{\gamma} u)(t):=\int_{-\infty}^{t}(t-s)^{\gamma -1}P_{\gamma}(t-s)f(s,u(s))\, ds,\ t\in {\mathbb R}.
$$
\end{defn}

Suppose that $M>0$ denotes the constant from (A), and the sequence $(M_{n})$ is defined through \eqref{em-en-a}. 
We will first state the following analogues of \cite[Theorem 2.10.3-Theorem
2.10.4]{nova-mono} and \cite[Theorem 2.10.9-Theorem
2.10.10]{nova-mono}, which can be also clarified for weighted pseudo-almost periodicity.

\begin{thm}\label{be-spore}
Let
$\rho_{1},\ \rho_{2} \in {\mathbb U}_{T}$ and $1< p<\infty.$ Suppose that  \emph{(P)} holds, $\beta>\theta >1-\beta$ and the following conditions hold:
\begin{itemize}
\item[(i)]
$f : {\mathbb R}\times Y\rightarrow X$ is weighted $S^{p}$-pseudo almost automorphic, 
$
f(t,y)=g(t,y)+q(t,y),
$ $t\in {\mathbb R},$ where $g(\cdot,\cdot)$ is $S^{p}$-almost automorphic and $q(\cdot,\cdot)$ satisfies \eqref{stepa-auto-durak-two},
uniformly on bounded subsets of $Y.$ 
\item[(ii)] Assume that  $ r\geq \max (p, p/p -1),$ $r>p/p-1$ and there exist two Stepanov $r$-almost automorphic scalar-valued functions $L_{f}(\cdot)$ and $L_{g}(\cdot)$ such that \eqref{minorca}-\eqref{minorca-ibiza} hold.
Set $q:=\frac{pr}{p+r}$ and $q':=\frac{pr}{pr-p-r}.$
\end{itemize}
If
$
q'(\beta -\theta-1)>-1
$
and $M_{n}<1$ for some $n\in {\mathbb N},$ then there exists an weighted pseudo-almost automorphic mild solution of inclusion \emph{\eqref{favini}}. The uniqueness of mild solutions holds in the case that ${\mathcal A}$ is single-valued.
\end{thm}

\begin{thm}\label{stepa-vuk-ght-auto-spore}
Let
$\rho_{1},\ \rho_{2} \in {\mathbb U}_{T}$ and and $1< p<\infty.$
Suppose that  \emph{(P)} holds, $\beta>\theta >1-\beta$ and the following conditions hold:
\begin{itemize}
\item[(i)]
$f : {\mathbb R}\times Y\rightarrow X$ is weighted $S^{p}$-pseudo almost automorphic, 
$
f(t,y)=g(t,y)+q(t,y),
$ $t\in {\mathbb R},$ where $g(\cdot,\cdot)$ is $S^{p}$-almost automorphic and $q(\cdot,\cdot)$ satisfies \eqref{stepa-auto-durak-two},
uniformly on bounded subsets of $Y.$ 
\item[(ii)] There exist two finite constants $L_{f}>0$ and $L_{g}>0$ such that
\eqref{minorcas}-\eqref{minorcas-ibiza} hold.
\end{itemize}
If
$
\frac{p}{p-1}(\beta -\theta-1)>-1
$
and $M_{n}<1$ for some $n\in {\mathbb N},$
then there exists a weighted pseudo-almost automorphic mild solution of inclusion \emph{\eqref{favini}}. The uniqueness of mild solutions holds provided that ${\mathcal A}$ is single-valued, additionally.
\end{thm}

\begin{thm}\label{be-spore-brze}
Let
$\rho_{1},\ \rho_{2} \in {\mathbb U}_{T}$ and $1< p<\infty.$ Suppose that  \emph{(P)} holds, $\beta>\theta >1-\beta$ and the following conditions hold:
\begin{itemize}
\item[(i)]
$f : {\mathbb R}\times Y\rightarrow X$ is weighted $S^{p}$-pseudo almost automorphic, 
$
f(t,y)=g(t,y)+q(t,y),
$ $t\in {\mathbb R},$ where $g(\cdot,\cdot)$ is $S^{p}$-almost automorphic and $q(\cdot,\cdot)$ satisfies \eqref{stepa-auto-durak-two},
uniformly on bounded subsets of $Y.$ 
\item[(ii)] Assume that  $ r\geq \max (p, p/p -1),$ $r>p/p-1$ and there exist two Stepanov $r$-almost automorphic scalar-valued functions $L_{f}(\cdot)$ and $L_{g}(\cdot)$ such that \eqref{minorca}-\eqref{minorca-ibiza} hold.
Set $q:=\frac{pr}{p+r}$ and $q':=\frac{pr}{pr-p-r}.$
\end{itemize}
If
$
q'(\gamma (\beta -\theta)-1)>-1
$
and $B_{n}<1$ for some $n\in {\mathbb N},$ then there exists an weighted pseudo-almost automorphic mild solution of inclusion \emph{\eqref{favini}}. The uniqueness of mild solutions holds in the case that ${\mathcal A}$ is single-valued.
\end{thm}

\begin{thm}\label{stepa-vuk-ght-auto-brze}
Let
$\rho_{1},\ \rho_{2} \in {\mathbb U}_{T}$ and and $1< p<\infty.$
Suppose that  \emph{(P)} holds, $\beta>\theta >1-\beta$ and the following conditions hold:
\begin{itemize}
\item[(i)]
$f : {\mathbb R}\times Y\rightarrow X$ is weighted $S^{p}$-pseudo almost automorphic, 
$
f(t,y)=g(t,y)+q(t,y),
$ $t\in {\mathbb R},$ where $g(\cdot,\cdot)$ is $S^{p}$-almost automorphic and $q(\cdot,\cdot)$ satisfies \eqref{stepa-auto-durak-two},
uniformly on bounded subsets of $Y.$ 
\item[(ii)] There exist two finite constants $L_{f}>0$ and $L_{g}>0$ such that
\eqref{minorcas}-\eqref{minorcas-ibiza} hold.
\end{itemize}
If
$
\frac{p}{p-1}(\gamma (\beta -\theta)-1)>-1
$
and $B_{n}<1$ for some $n\in {\mathbb N},$
then there exists a weighted pseudo-almost automorphic mild solution of inclusion \emph{\eqref{favini}}. The uniqueness of mild solutions holds provided that ${\mathcal A}$ is single-valued, additionally.
\end{thm}

The following theorem is an analogue of \cite[Theorem 2.12.5]{nova-mono}:

\begin{thm}\label{xia-composition}
Assume that $\rho_{1},\ \rho_{2}\in {\mathbb U}_{T},$ $1\leq p<\infty,$ 
and the following conditions hold:
\begin{itemize}
\item[(i)] $f \in PAA({\mathbb R} \times X , X,\rho_{1},\rho_{2})$ is weighted pseudo-almost automorphic.
\item[(ii)] The inequality \eqref{minorcas} holds with some bounded non-negative function $L_{f}(\cdot).$
\item[(iii)] $\sum_{n=1}^{\infty}M_{n}<\infty.$
\end{itemize}
Then there exists a unique weighted pseudo-almost automorphic solution of inclusion \emph{\eqref{favini}}.
\end{thm}

And, at the end of paper, a few words about situations in which we can apply our abstract results. As already mentioned in our previous researches, we can apply our results in the analysis of existence and uniqueness of weighted (asymptotically) almost automorphic solutions of the fractional Poisson heat equation
\[\left\{
\begin{array}{l}
D_{t,+}^{\gamma}[m(x)v(t,x)]=(\Delta -b )v(t,x) +f(t,x),\quad t\in {\mathbb R},\ x\in {\Omega};\\
v(t,x)=0,\quad (t,x)\in [0,\infty) \times \partial \Omega ,\\
\end{array}
\right.
\]
in the space $X:=L^{p}(\Omega),$ where $\Omega$ is a bounded domain in ${\mathbb R}^{n},$ $b>0,$ $m(x)\geq 0$ a.e. $x\in \Omega$, $m\in L^{\infty}(\Omega),$ $\gamma \in (0,1)$ and $1<p<\infty ,$ as well as its semilinear analogue
\[\left\{
\begin{array}{l}
D_{t,+}^{\gamma}[m(x)v(t,x)]=(\Delta -b )v(t,x) +f(t,m(x)v(t,x)),\ t\in {\mathbb R},\ x\in {\Omega};\\
v(t,x)=0,\quad (t,x)\in [0,\infty) \times \partial \Omega .\\
\end{array}
\right.
\]
Furthermore, we can analyze the existence and uniqueness of asymptotically $S^{p}$-almost automorphic solutions of the following fractional damped Poisson-wave type equation\index{equation!damped Poisson-wave} in the spaces $X:=H^{-1}(\Omega)$ or $X:=L^{p}(\Omega):$
\[\left\{
\begin{array}{l}
{\mathbf D}_{t}^{\gamma}\bigl( m(x){\mathbf D}_{t}^{\gamma}u\bigr)+\bigl(2\omega m(x)-\Delta \bigr){\mathbf D}_{t}^{\gamma}u+\bigl(A(x;D)-\omega \Delta+\omega^{2}m(x)\bigr)u(x,t)=f(x,t),\\ t\geq 0,\ x\in \Omega \ \ ; \ \
u={\mathbf D}_{t}^{\gamma}=0,\quad (x,t)\in \partial \Omega \times [0,\infty),\\
u(0,x)=u_{0}(x),\ m(x)\bigl[ {\mathbf D}_{t}^{\gamma}u(x,0)+\omega u_{0}\bigr]=m(x)u_{1}(x),\quad x\in {\Omega},
\end{array}
\right.
\]
where $\Omega \subseteq {\mathbb R}^{n}$ is a bounded open domain with smooth boundary, $1<p<\infty ,$ $m(x)\in L^{\infty}(\Omega),$ $m(x) \geq 0$ a.e. $x\in \Omega,$
$\Delta$ is the Dirichlet Laplacian in $ L^{2}(\Omega),$ acting with domain $H^{1}_{0}(\Omega) \cap H^{2}(\Omega),$
and $A(x;D)$ is a second order linear differential operator on $\Omega$ with coefficients continuous on $\overline{\Omega}$ and  the Caputo fractional derivative ${\mathbf D}_{t}^{\gamma}$ is taken in a slightly weakened sense \cite{afid-mlo};
see \cite[Example 6.1]{faviniyagi} and \cite[Example 6.2]{element} for more details. 

As already mentioned in the introductory part, our results can be applied to almost sectorial operators so that we can 
examine the existence and uniqueness of weighted pseudo-almost automorphic solutions for
certain classes of higher order (semilinear) elliptic differential equations in H\"older spaces (see e.g. W. von Wahl \cite{wolf}). Our results are also applicable in the study of generalized weighted pseudo-almost periodic solutions and generalized weighted pseudo-almost automorphic solutions of the following abstract integral inclusion
\begin{align*}
u(t)\in {\mathcal A}\int^{t}_{\infty}a(t-s)u(s)\, ds +f(t),\ t\in {\mathbb R}
\end{align*}
where $a \in L_{loc}^{1}([0,\infty)),$ $a\neq 0,$ $f : {\mathbb R} \rightarrow X$ satisfies certain assumptions and ${\mathcal A}$ is a closed multivalued linear operator on $X,$ as well as its semilinear analogue
\begin{align*}
u(t)\in {\mathcal A}\int^{t}_{\infty}a(t-s)u(s)\, ds +f(t,u(t)),\ t\in {\mathbb R}
\end{align*}
(see e.g. the papers by C. Cuevas, C. Lizama \cite{cuevas-lizama-duo} and H. R. Henr\' iquez, C. Lizama \cite{hernan-lizama}). The main results of paper \cite{abbas-indian} can be also fomulated for inclusions.

In \cite{gaston-marko}, G. M. N'Gu\' er\' ekata and M. Kosti\' c have recently studied various classes of generalized almost periodic and 
generalized almost automorphic solutions of abstract multi-term fractional differential inclusions in Banach spaces. We close the paper with the observation that Proposition \ref{Cauchy} can be applied in the qualitative analysis of the abstract multi-term Cauchy inclusion
\begin{align*}
\Bigl[u(\cdot)&- \bigl(
g_{\zeta+1+i}\ast f\bigr)(\cdot)Cx\Bigr]  +\sum
\limits_{j=1}^{n-1}c_{j}g_{\alpha_{n}-\alpha_{j}}\ast \Bigl[
u(\cdot)- \bigl(g_{\zeta+1+i}\ast f\bigr)(\cdot)Cx \Bigr]
\\ &+\sum \limits_{j\in {{\mathbb N}_{n-1}}\setminus D_{i}}c_{j}\bigl[g_{\alpha_{n}-\alpha_{j}+i+\zeta+1}\ast
f\bigr](\cdot)Cx \in 
{\mathcal A}\bigl[g_{\alpha_{n}-\alpha}\ast u\bigr](\cdot),
\end{align*}
where $0 \leq \alpha_{1}<\cdot \cdot
\cdot<\alpha_{n},$ $0\leq \alpha<\alpha_{n},$ $x\in X,$ $C\in L(X)$ is injective, $c_{j}\in {\mathbb C}$ for $1\leq j\leq n-1,$ $C{\mathcal A}\subseteq {\mathcal A}C,$ $f\in L_{loc}^{1}([0,\infty) : X)$ and the set $D_{i}$ has the same meaning as in \cite{gaston-marko} ($0\leq i \leq \lceil \alpha_{n}\rceil-1$).

\end{document}